\documentclass[12pt]{amsart}

\usepackage{amssymb, amsfonts}
\usepackage{url}
\usepackage[all,arc]{xy}
\usepackage{amscd}
\usepackage{mathrsfs}
\usepackage{geometry}
\usepackage[colorlinks,urlcolor=blue]{hyperref}
\usepackage{comment}
\usepackage{enumerate}
\usepackage{graphicx}
\usepackage{bm}
\usepackage{color}
\usepackage{anyfontsize}
\usepackage{caption}

\numberwithin{equation}{section}

\theoremstyle{plain}
\newtheorem{theorem}{Theorem}[section]
\newtheorem{corollary}[theorem]{Corollary}
\newtheorem{lemma}[theorem]{Lemma}
\newtheorem{proposition}[theorem]{Proposition}
\newtheorem{conjecture}[theorem]{Conjecture}

\theoremstyle{remark}
\newtheorem{remark}{Remark}[section]

\begin{document}

\date{\today} 

\title[Degenerating sequences of conformal classes]{Degenerating sequences of conformal classes and the conformal Steklov spectrum}


\author{Vladimir Medvedev}

\address{D\'{e}partement de Math\'{e}matiques et de Statistique, Pavillon Andr\'{e}-Aisenstadt, Universit\'{e} de Montr\'{e}al, Montr\'{e}al, QC, H3C 3J7, Canada \newline {\em and}
\newline Peoples' Friendship University of Russia (RUDN University), 6 Miklukho-Maklaya Street, Moscow, 117198, Russian Federation \newline {\em and}
\newline Faculty of Mathematics, National Research University Higher School of Economics, 6 Usacheva Street, Moscow, 119048, Russian Federation}

\email{medvedevv@dms.umontreal.ca}




\begin{abstract}
Let $\Sigma$ be a compact surface with boundary. For a given conformal class $c$ on $\Sigma$ the functional $\sigma_k^*(\Sigma,c)$ is defined as the supremum of the $k-$th normalized Steklov eigenvalue over all metrics in $c$. We consider the behaviour of this functional on the moduli space of conformal classes on $\Sigma$. A precise formula for the limit of $\sigma_k^*(\Sigma,c_n)$ when the sequence $\{c_n\}$ degenerates is obtained. We apply this formula to the study of natural analogs of the Friedlander-Nadirashvili invariants of closed manifolds defined as $\inf_{c}\sigma_k^*(\Sigma,c)$, where the infimum is taken over all conformal classes $c$ on $\Sigma$. We show that these quantities are equal to $2\pi k$ for any surface with boundary. As an application of our techniques we obtain new estimates on the $k-$th normalized Steklov eigenvalue of a non-orientable surface in terms of its genus and the number of boundary components. 
\end{abstract}

\maketitle


\newcommand\cont{\operatorname{cont}}
\newcommand\diff{\operatorname{diff}}

\newcommand{\dvol}{\text{dA}}
\newcommand{\GL}{\operatorname{GL}}
\newcommand{\myO}{\operatorname{O}}
\newcommand{\myP}{\operatorname{P}}
\newcommand{\eye}{\operatorname{Id}}
\newcommand{\myF}{\operatorname{F}}
\newcommand{\Vol}{\operatorname{Vol}}
\newcommand{\odd}{\operatorname{odd}}
\newcommand{\even}{\operatorname{even}}
\newcommand{\ol}{\overline}
\newcommand{\mye}{\operatorname{E}}
\newcommand{\myo}{\operatorname{o}}
\newcommand{\myt}{\operatorname{t}}
\newcommand{\irr}{\operatorname{Irr}}
\newcommand{\mydiv}{\operatorname{div}}
\newcommand{\re}{\operatorname{Re}}
\newcommand{\im}{\operatorname{Im}}
\newcommand{\can}{\operatorname{can}}
\newcommand{\scal}{\operatorname{scal}}
\newcommand{\tr}{\operatorname{trace}}
\newcommand{\sgn}{\operatorname{sgn}}
\newcommand{\SL}{\operatorname{SL}}
\newcommand{\myspan}{\operatorname{span}}
\newcommand{\mydet}{\operatorname{det}}
\newcommand{\SO}{\operatorname{SO}}
\newcommand{\SU}{\operatorname{SU}}
\newcommand{\specl}{\operatorname{spec_{\mathcal{L}}}}
\newcommand{\fix}{\operatorname{Fix}}
\newcommand{\id}{\operatorname{id}}
\newcommand{\grad}{\operatorname{grad}}
\newcommand{\singsup}{\operatorname{singsupp}}
\newcommand{\wave}{\operatorname{wave}}
\newcommand{\ind}{\operatorname{ind}}
\newcommand{\mynull}{\operatorname{null}}
\newcommand{\inj}{\operatorname{inj}}
\newcommand{\arcsinh}{\operatorname{arcsinh}}
\newcommand{\Spec}{\operatorname{Spec}}
\newcommand{\Ind}{\operatorname{Ind}}
\newcommand{\floor}[1]{\left \lfloor #1  \right \rfloor}

\newcommand\restr[2]{{
  \left.\kern-\nulldelimiterspace 
  #1 
  \vphantom{\big|} 
  \right|_{#2} 
  }}


\section{Introduction and main results}

Let $(\Sigma,g)$ be a compact Riemannian surface with boundary. In this paper we always assume that $\Sigma$ is connected and the boundary of $\Sigma$ is non-empty and smooth. Consider \textit{the Steklov problem} defined in the following way
\begin{gather*}
\begin{cases}
\Delta u=0&\text{in $\Sigma$},\\
\frac{\partial u}{\partial n}=\sigma u&\text{on $\partial \Sigma$},
\end{cases}
\end{gather*}
where $\Delta=-\mydiv_g \circ \grad_g$ is the Laplace-Beltrami operator and $\frac{\partial}{\partial n}$ is the outward unit normal vector field along the boundary. The collection of all numbers $\sigma$ for which the Steklov problem admits a solution is called the \textit{Steklov spectrum} of the surface $\Sigma$. The Steklov spectrum is a discrete set of real numbers called Steklov eigenvalues with finite multiplicities satisfying the following condition (see e.g. \cite{girouard2017spectral})

\begin{align*}
0=\sigma_0(g) < \sigma_1(g) \leq \sigma_2(g) \leq\ldots\nearrow +\infty.
\end{align*}

The Steklov spectrum enables us to define the following homothety-invariant functional on the set $\mathcal{R}(\Sigma)$ of Riemannian metrics on $\Sigma$
\begin{align*}
\overline{\sigma}_k(\Sigma,g):=\sigma_k(g)L_g(\partial \Sigma),
\end{align*}
where $L_g(\partial \Sigma)$ stands for the length of the boundary of $\Sigma$ in the metric $g$. The functional $\overline{\sigma}_k(\Sigma,g)$ is called the \textit{$k-$th normalized Steklov eigenvalue}. It was shown in~\cite{colbois2011isoperimetric} (see also~\cite{hassannezhad2011, kokarev2014variational}) that if $\Sigma$ is an orientable surface then the functional $\overline{\sigma}_k(\Sigma,g)$ is bounded from above. Moreover, the following theorem holds


\begin{theorem}[\cite{girouard2012upper}]
\label{Kor}
Let $(\Sigma,g)$ be a compact orientable surface of genus $\gamma$ with $l$ boundary components. Then one has
\begin{gather*}
\overline{\sigma}_k(\Sigma,g) \leq 2\pi k(\gamma+l).
\end{gather*}
\end{theorem}

In this paper we prove that a similar estimate holds for non-orientable surfaces.

\begin{theorem}\label{non-bound}
Let $\Sigma$ be a compact non-orientable surface of genus $\gamma$ with $l$ boundary components. Then one has 
\begin{gather*}
\overline{\sigma}_k(\Sigma,g) \leq 4\pi k(\gamma+2l).
\end{gather*}
\end{theorem}
Here the genus of a non-orientable surface is defined as the genus of its orientable cover. 

\begin{remark}
The estimate in Theorem~\ref{Kor} has been improved in~\cite{karpukhin2015bounds} by a bound which is linear in $k+\gamma+l$ instead of $k(\gamma+l)$. However, the proof of this result uses orientability in an essential way, see \cite[Section 6]{karpukhin2015bounds}. It would be interesting to obtain a similar improvement in Theorem~\ref{non-bound}.
\end{remark}

Theorems~\ref{Kor} and~\ref{non-bound} enable us to define the following functionals
$$
\sigma^*_k(\Sigma):=\sup_{\mathcal{R}(\Sigma)} \overline{\sigma}_k(\Sigma,g),
$$
and
$$
\sigma^*_k(\Sigma,[g]):=\sup_{[g]} \overline{\sigma}_k(\Sigma,g).
$$

\begin{remark}
Note that we cannot define the functionals $\sigma^*_k(\Sigma)$ and $\sigma^*_k(\Sigma,[g])$ in higher dimensions. Indeed, it was proved in the paper~\cite{colbois2019compact} that if $n=\dim M \geq 3$ then the functional $\overline{\sigma}_k(M,g):=\sigma_k(g)Vol(\partial M, g)^{1/(n-1)}$, where $Vol(\partial M, g)$ denotes the volume of the boundary with respect to the metric $g$, is not bounded from above on the set of Riemannian metrics $\mathcal R(M)$. Moreover, it is not even bounded from above in the conformal class $[g]$.  
\end{remark}

The functional $\sigma^*_k(\Sigma)$ is an object of intensive research during the last decade (see e.g. \cite{fraser2011first, fraser2016sharp, colbois2016steklov, petrides2019maximizing, girouard-lagace, matpet}). 



The functional $\sigma^*_k(\Sigma,[g])$ which is called the \textit{$k-$th conformal Steklov eigenvalue} is less studied. Let us mention some results concerning $\sigma^*_k(\Sigma,[g])$. First since the disc admits the unique conformal structure one can conclude that $\sigma^*_k(\mathbb D^2,[g_{can}])=\sigma^*_k(\mathbb D^2),$ where $g_{can}$ stands for the Euclidean metric on $\mathbb D^2$ with unit boundary length. The value of $\sigma^*_k(\mathbb D^2)$ is known: $\sigma^*_k(\mathbb D^2)=2\pi k$ (see \cite{weinstock1954inequalities} for $k=1$ and \cite{girouard2010hersch} for all $k \geq 1$). Let us also mention the resent paper \cite{fraser2020some}, where the authors particularly obtain new results about the functional $\sigma^*_k(\mathbb D^2)$. 

The functional $\sigma^*_k(\Sigma,[g])$ is the main research object of the paper \cite{petrides2019maximizing}.

\begin{theorem}[\cite{petrides2019maximizing}]\label{petridesmax}
For every Riemannian metric $g$ on a compact surface $\Sigma$ with boundary one has
\begin{align}\label{petya}
\sigma^*_k(\Sigma,[g]) \geq \sigma^*_{k-1}(\Sigma,[g])+\sigma^*_1(\mathbb{D}^2, [g_{can}]),
\end{align}
particularly
\begin{align}\label{el soufi}
\sigma^*_k(\Sigma,[g]) \geq 2\pi k.
\end{align}
Moreover, if the inequality~\eqref{petya} is strict then there exists a Riemannian metric $\tilde g\in [g]$ such that $\overline\sigma_k(\Sigma,\tilde g)=\sigma^*_k(\Sigma,[g])$. 
\end{theorem}
New interesting results about the functional $\sigma^*_k(\Sigma,[g])$ were recently obtained in the paper~\cite{karpukhin-stern}.

\begin{remark}
The result analogous to Theorem~\ref{petridesmax} for the conformal spectrum of the Laplace-Beltrami operator on closed surfaces also holds (see \cite{nadirashvili2015conformal, MR3438833, petrides2014existence, petrides2018existence, karpukhin2020conformally}). For further information concerning the spectrum of the Laplace-Beltrami operator on closed surfaces see the surveys~\cite{MR3203194, MR4017613} and references therein.
\end{remark}

It is easy to see that the connection between the functionals $\sigma^*_k(\Sigma)$ and $\sigma^*_k(\Sigma,[g])$ is expressed by the formula  
\begin{align*}
\sigma^*_k(\Sigma)=\sup_{[g]} \sigma^*_k(\Sigma,[g]).
\end{align*}
One can ask what do we get if we replace $\sup_{[g]}$ by $\inf_{[g]}$ in this formula? In this case we get the following quantity 
\begin{align*}
I^\sigma_k(\Sigma):=\inf_{[g]} \sigma^*_k(\Sigma,[g]),
\end{align*}
It is an analog of the Friedlander-Nadirashvili invariant of closed manifolds. The first Friedlander-Nadirashvili invariant of a closed manifold was introduced in the paper \cite{MR1717641} in 1999. The $k-$th Nadirashvili-Friedlander invariant of a closed surface has been recently studied in the paper \cite{karpukhin2019friedlander}.




\begin{figure}[h!]
  \centering
  \def\svgwidth{\columnwidth}
  \includegraphics[scale=0.5]{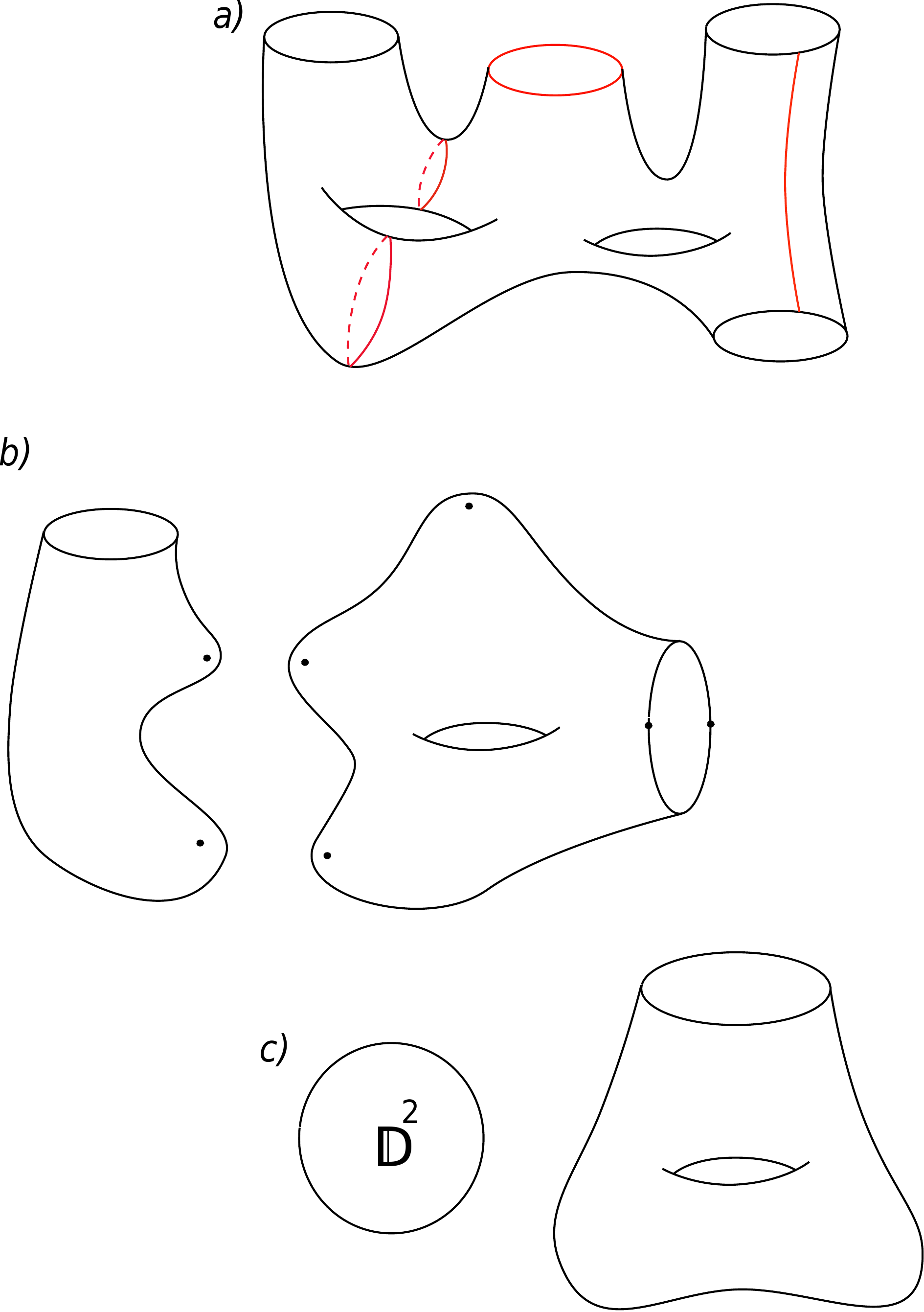}
    \footnotesize
    \caption{
   An example of a degenerating sequence of conformal classes $\{c_n\}$ on a surface $\Sigma$ of genus $2$ with $4$ boundary components. $a)$ The \textit{red} curves correspond to collapsing geodesics for the sequence of metrics of constant Gauss curvature and geodesic boundary $\{h_n\}, ~h_n\in c_n$ corresponding to the degenerating sequence of conformal classes $\{c_n\}$. $b)$ The compactified limiting space $\widehat{\Sigma_\infty}$ (see Section~\ref{geometry}). The black points correspond to the points of compactification. $c)$ The surface $\widehat{\Sigma_\infty}$ is homeomorphic to the disjoint union of a disc and a surface of genus $1$ with $1$ boundary component. 
    }
    \label{F}
\end{figure}

In the study of functionals like $\sigma^*_k(\Sigma)$ and $I^\sigma_k(\Sigma)$ one considers maximizing and minimizing sequences of conformal classes $\{c_n\}$ on the \textit{moduli space of conformal classes on $\Sigma$}, i.e. $\sigma^*_k(\Sigma,c_n) \to \sigma^*_k(\Sigma)$ or $\sigma^*_k(\Sigma,c_n) \to I^\sigma_k(\Sigma)$ as $n\to\infty$. Due to the Uniformization theorem conformal classes on $\Sigma$ are in one-to-one correspondence (up to an isometry) with metrics on $\Sigma$ of constant Gauss curvature and geodesic boundary. Therefore, any sequence of conformal classes $\{c_n\}$ on $\Sigma$ corresponds to a sequence of Riemannian surfaces of constant Gauss curvature and geodesic boundary $\{(\Sigma,h_n)\},~h_n\in c_n$ and we can consider the moduli space of conformal classes on $\Sigma$ as the set of all $(\Sigma,h)$, where $h$ is a metric of constant Gauss curvature and geodesic boundary, endowed with $C^\infty-$topology (see Section~\ref{geometry}). Note that the moduli space of conformal structures is a non-compact topological space. For any sequence $\{c_n\}$ there are two possible scenarios: either this sequence remains in a compact part of the moduli space or it escapes to infinity. Let $(\Sigma_\infty, c_\infty)$ denote the \textit{limiting space}, i.e. $(\Sigma_\infty, c_\infty)=\lim_{n\to\infty}(\Sigma,c_n)$. We compactify $\Sigma_\infty$ if necessary. Let $\widehat{\Sigma_\infty}$ denote the compactified limiting space. It turns out that if the first scenario realizes then we get $\widehat{\Sigma_\infty}=\Sigma$ and $c_\infty$ is a genuine conformal class on $\Sigma$ for which the value $\sigma^*_k(\Sigma)$ or $I^\sigma_k(\Sigma)$ is attained. If the second scenario realizes then we say that the sequence $\{c_n\}$ \textit{degenerates}. It turns out that in this case there exists a finite collection of pairwise disjoint geodesics for the metrics $h_n$ whose lengths in $h_n$ tend to $0$ as $n$ tends to $\infty$. We refer to these geodesics as \textit{pinching} or \textit{collapsing}. They can be of the following three types: the collapsing boundary components, the collapsing geodesics with no self-intersection crossing the boundary $\partial\Sigma$ at two points and the collapsing geodesics with no self-intersection which do not cross $\partial\Sigma$. Note that in this case the topology of $\Sigma$ necessarily changes when we pass to the limit as $n\to\infty$, i.e. the compact surfaces $\widehat\Sigma_\infty$ and $\Sigma$ are of different topological types. In particular, the surface $\widehat\Sigma_\infty$ can be disconnected (see Figure~\ref{F}). We refer to Section~\ref{geometry} for more details. 

The following theorem establishes the correspondence between $\sigma^*_k(\widehat\Sigma_\infty,c_\infty)$ and the limit of $\sigma^*_k(\Sigma,c_n)$ when the sequence of conformal classes $c_n$ degenerates (see Section~\ref{geometry} for the definition). It is an analog of \cite[Theorem 2.8]{karpukhin2019friedlander} for the Steklov setting.

\begin{theorem}\label{conf&conv}
Let $\Sigma$ be a compact surface of genus $\gamma$ with $l>0$ boundary components and let $c_n\to c_\infty$ be a degenerating sequence of conformal classes. Consider the corresponding sequence $\{h_n\}$ of metrics of constant Gauss curvature and geodesic boundary. Suppose that there exist $s_1$ collapsing boundary components and $s_2$ collapsing geodesics with no self-intersection which cross the boundary at two points. Moreover, suppose that $\widehat{\Sigma_{\infty}}$ has $m$ connected components $\Sigma_{\gamma_{i},l_i}$ of genus $\gamma_i$ with $l_i>0$ boundary components, $\gamma_i+l_i<\gamma+l$, $i=1,\ldots,m$. Then one has

\begin{equation}
\label{deg_limit}
\begin{split}
\lim_{n \to \infty} \sigma^*_k (\Sigma, c_n)= \max \Big(\sum^{m}_{i=1} \sigma^*_{k_i}(\Sigma_{\gamma_{i},l_i}, c_\infty)+\sum_{i=1}^{s_1+s_2}\sigma^*_{r_i}(\mathbb{D}^2)\Big),
\end{split}
\end{equation}
where the maximum is taken over all possible combinations of indices such that
 $$
 \sum_{i=1}^{m} k_i + \sum_{i=1}^{s_1+s_2} r_i = k.
 $$
\end{theorem}
 
\begin{remark}\label{main_remark}
Let $\Sigma$ denote either cylinder or the M\"obius band. Theorem~\ref{conf&conv} particularly implies that if the sequence of conformal classes $\{c_n\}$ on $\Sigma$ degenerates then we necessarily have:
$$
\lim_{n \to \infty} \sigma^*_k (\Sigma, c_n)=2\pi k.
$$
\end{remark} 

\begin{remark}
In Theorem~\ref{conf&conv} the sequence $\{h_n\}$ can also have collapsing geodesics not crossing the boundary of $\Sigma$. Moreover, it can happen that the limiting space $\widehat{\Sigma_{\infty}}$ has \textit{closed} components (see Figure~\ref{F'}). Anyway, in Theorem~\ref{conf&conv} we take only components of $\widehat{\Sigma_{\infty}}$ which have non-empty boundary. 
\end{remark}

\begin{figure}[h!]
  \centering
  \def\svgwidth{\columnwidth}
  \includegraphics[scale=0.45]{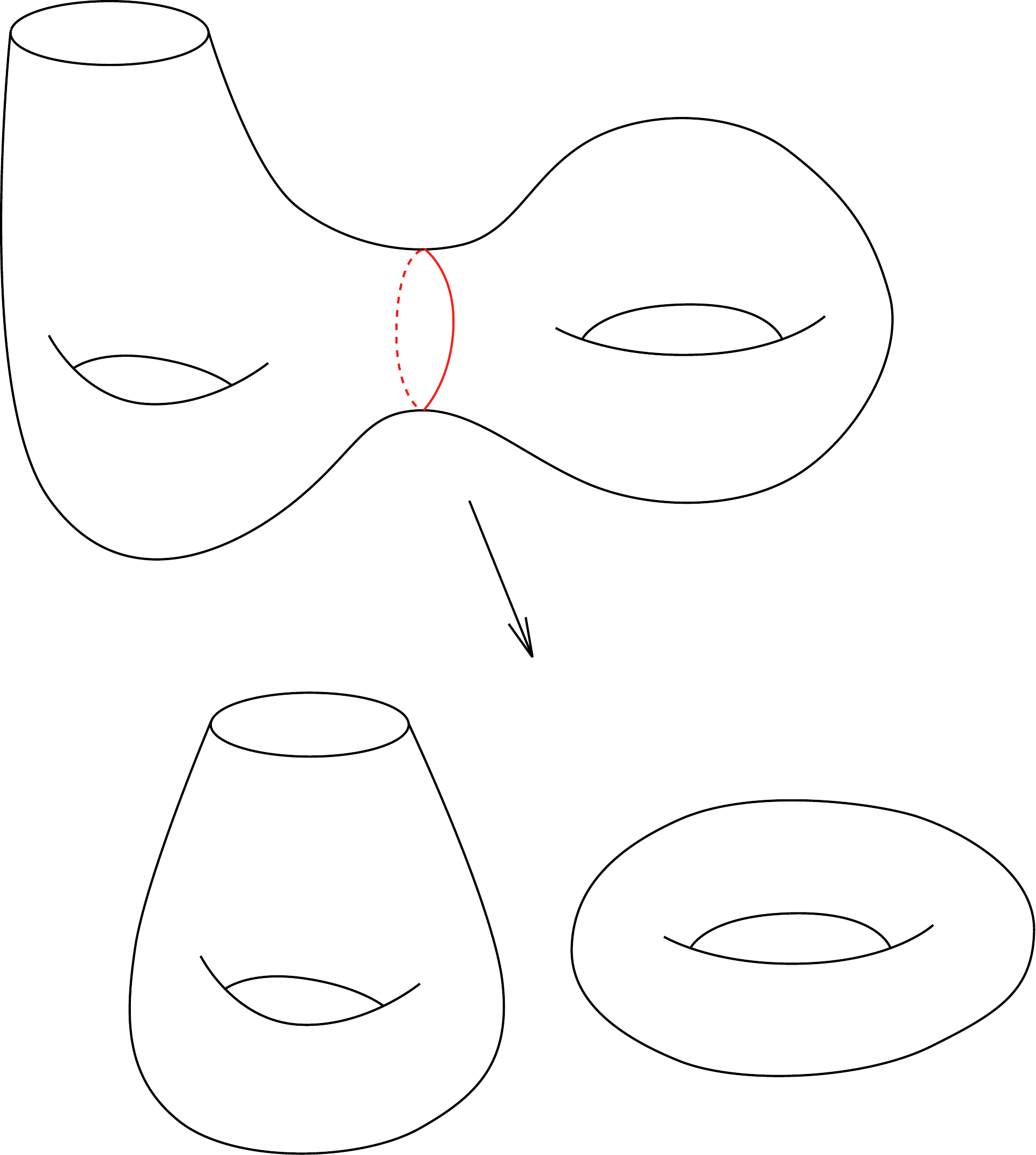}
    \footnotesize
    \caption{
   An example of a degenerating sequence of conformal classes $\{c_n\}$ on a surface of genus $2$ with $1$ boundary components such that the limiting space contains a closed component. In Theorem~\ref{conf&conv} we take only the component on the left which has non-empty boundary. Note that in this case $s_1=s_2=0$.
    }
    \label{F'}
\end{figure}


The main tool that we use in the proof of Theorem~\ref{conf&conv} is the \textit{Steklov-Neumann boundary problem} also known as the \textit{sloshing problem}. Let $\Omega$ be a Lipschitz domain in $(\Sigma,g)$ such that $\overline \Omega \cap \partial \Sigma = \partial^S\Omega \neq \O$. Let $\partial^N\Omega=\partial \Omega \setminus \partial \Sigma$. Then the Steklov-Neumann problem is defined as:
\begin{gather}\label{SN}
\begin{cases}
\Delta_g u=0&\text{in $\Omega$},\\
\frac{\partial u}{\partial n}=0&\text{on $\partial^N\Omega$},\\
\frac{\partial u}{\partial n}=\sigma^Nu&\text{on $\partial^S\Omega$}.
\end{cases}
\end{gather}
The numbers $\sigma^N$ for which the Steklov-Neumann problem admits a solution are called \textit{Steklov-Neumann eigenvalues}. It is known (see \cite{banuelos2010eigenvalue} and references therein) that the set of Steklov-Neumann eigenvalues is not empty and discrete
\begin{align*}
0=\sigma^N_0(g) < \sigma^N_1(g) \leq \sigma^N_2(g) \leq\ldots\nearrow +\infty.
\end{align*}
Every Steklov-Neumann eigenvalue admits the following variational characterization:
\begin{gather}\label{charSN}
\sigma^N_k(g)=\inf_{V_k\subset \mathcal H^1(\Omega)}\sup_{0 \neq u \in V_k}\frac{\int_\Omega|\nabla u|^2dv_g}{\int_{\partial^S\Omega}u^2ds_g},
\end{gather}
where the infimum is taken over all $k-$dimensional subspaces of the space $\mathcal H^1(\Omega)=\{u\in H^1(\Omega,g)~|~\int_{\partial^S\Omega} uds_g=0\}$.

Similarly to the case of the Steklov problem we define normalized Steklov-Neumann eigenvalues as
$$
\overline\sigma^N_k(\Omega, \partial^S\Omega,g):=\sigma^N_k(g)L_g(\partial^S\Omega).
$$
In this notation we always indicate the Steklov part of the boundary at the second place. Sometimes we also use the notation $\sigma^N_k(\Omega, \partial^S\Omega,g)$ for $\sigma^N_k(\Omega,g)$ to emphasize that the Steklov boundary condition is imposed on $\partial^S\Omega$.
\begin{remark}
Consider $\Omega$ as a surface with Lipschitz boundary. It also follows from~\cite[Theorem $A_k$]{kokarev2014variational} that the quantity $\overline\sigma^N_k(\Omega, \partial^S\Omega,g)$ is bounded from above on $[g]$ and we can define the invariant $\sigma^{N*}_k(\Omega, \partial^S\Omega,[g])$ in the same way as the invariant $\sigma^{*}_k(\Sigma,[g])$.
\end{remark}




Theorem~\ref{conf&conv} enables us to establish the value of $I^\sigma_k$.

\begin{theorem}
\label{disproof}
 Let $\Sigma$ be a compact surface with boundary. Then one has $I^\sigma_k(\Sigma)=I^\sigma_k(\mathbb D^2)=2\pi k$. 
  \end{theorem}





\subsection{Discussion}

Let us discuss the estimate obtained in Theorem~\ref{non-bound}. The first estimate on $\overline{\sigma}_1(\Sigma,g)$ where $\Sigma$ is a non-orientable surface of genus $\gamma$ with boundary was obtained in the paper \cite{MR3167132}. It reads
$$
\overline{\sigma}_1(\Sigma,g) \leq 24\pi (\gamma+1),
$$
if $\gamma\geq 1$ and 
$$
\overline{\sigma}_1(\Sigma,g) \leq 12\pi,
$$
if $\gamma=0$. Moreover, it follows from the papers~\cite{kokarev2014variational, MR3579963} that
\begin{gather}\label{Kokarev}
\overline{\sigma}_1(\Sigma,g) \leq 16\pi \Big[\frac{\gamma+3}{2}\Big],
\end{gather}
were $[x]$ stands for the integer part of the number $x$.

Very recently in the paper~\cite{karpukhin-stern} estimate~\eqref{Kokarev} has been improved and extended for $k=2$: consider $\Sigma$ as a domain with smooth boundary on a closed surface $M$, then one has
\begin{gather}\label{Karpukhin}
\overline{\sigma}_k(\Sigma,g) \leq \Lambda_k(M),~k=1,2.
\end{gather}
In this estimate $\Lambda_k(M):=\sup_{g\in \mathcal R(M)} \lambda_k(g)\Vol(M,g)$, where $\lambda_k(g)$ is the $k-$th Laplace eigenvalue of the metric $g$, $\Vol(M,g)$ is the volume of $M$ in the metric $g$ and $\mathcal R(M)$ is the set of Riemannian metrics on $M$. Note that estimate~\eqref{Karpukhin} does not depend on the number of boundary components. Combining estimate~\eqref{Karpukhin} with our estimate we get
$$
\overline{\sigma}_k(\Sigma,g) \leq \min\{\Lambda_k(M), 4\pi k(\gamma+2l)\},~k=1,2.
$$
Particularly, for the M\"obius band one has
$$
\overline{\sigma}_k(\mathbb{MB},g) \leq \min\{\Lambda_k(\mathbb{RP}^2), 8\pi k\},~k=1,2,
$$
since $\mathbb{MB} \subset \mathbb{RP}^2$. The value $\Lambda_k(\mathbb{RP}^2)$ is known for all $k$ (see~\cite{karpukhin2019index}): $\Lambda_k(\mathbb{RP}^2)=4\pi(2k+1)$. Hence
$$
\overline{\sigma}_k(\mathbb{MB},g) \leq \min\{4\pi(2k+1), 8\pi k\}=8\pi k,~k=1,2.
$$
In the paper~\cite{fraser2016sharp} it was shown that $\overline{\sigma}_1(\mathbb{MB},g) \leq 2\pi\sqrt{3}$ which is obviously $\leq 8\pi$. 

We proceed with the discussion of the functional $I^\sigma_k$. Unlike Theorem 1.4 in \cite{karpukhin2019friedlander} Theorem \ref{disproof} says nothing about conformal classes on which the value $I^\sigma_k(\Sigma)$ is attained. We conjecture that

\begin{conjecture}\label{Conj1}
The infimum $I^\sigma_k(\Sigma)$ is attained if and only if $\Sigma$ is diffeomorphic to the disc $\mathbb D^2$. 
\end{conjecture}

Note that this conjecture would be a corollary of the following one 

\begin{conjecture}\label{Conj2}
Let $\Sigma$ be a compact surface non-diffeomorphic to the disc. Then for every conformal class $c$ on $\Sigma$ one has
$$
\sigma^*_1(\Sigma,c)>\sigma^*_1(\mathbb D^2)=2\pi.
$$ 
\end{conjecture}
This conjecture is an analog of the Petrides rigidity theorem for the first conformal Laplace eigenvalue \cite[Theorem 1]{petrides2014existence}. Recently this conjecture has been confirmed in the case of the cylinder and the M\"obius band (see \cite{matthiesen2020remark}).  We plan to tackle Conjectures~\ref{Conj1} and~\ref{Conj2} in the subsequent papers.

Let us discuss the analogy between the quantity $I^\sigma_k$ and the Friedlander-Nadirashvili invariant of closed surfaces $I_k$. In the paper~\cite{karpukhin2019friedlander} it was conjectured that $I_k$ are invariants of cobordisms of closed surfaces (see Conjecture 1.8). Similarly, one can see that $I^\sigma_k$ are invariants of cobordisms of compact surfaces with boundary. Let us recall that two compact surfaces with boundary $(\Sigma_1,\partial\Sigma_1)$ and $(\Sigma_2,\partial\Sigma_2)$ are called cobordant if there exists a 3-dimensional \textit{manifold with corners} $\Omega$ whose boundary is $\Sigma_1\cup_{\partial\Sigma_1} W \cup_{\partial\Sigma_2}\Sigma_2$, where $W$ is a cobordism of $\partial\Sigma_1$ and $\partial\Sigma_2$ (i.e. $W$ is a surface with boundary $\partial\Sigma_1\sqcup\partial\Sigma_2$). Following~\cite{borodzik2016morse} we denote a cobordism of two surfaces $(\Sigma_1,\partial\Sigma_1)$ and $(\Sigma_2,\partial\Sigma_2)$ by $(\Omega;\Sigma_1,\Sigma_2,W;\partial\Sigma_1,\partial\Sigma_2)$. One can easily see that the cobordisms of surfaces with boundary are trivial. Indeed, we can construct the following cobordism of a surface $(\Sigma,\partial\Sigma)$ and $(\O, \O)$: $(\Sigma\times [0,1];\Sigma\times \{0\}, \O,\partial\Sigma\times [0,1]\cup\Sigma\times \{1\};\partial\Sigma, \O)$. A fundamental fact about cobordisms of surfaces with boundary is \textit{Theorem about splitting cobordisms} (see~\cite[Theorem 4.18]{borodzik2016morse}) which says that every cobordism of compact surfaces with boundary can be split into a sequence of cobordisms given by a handle attachment and cobordisms given by a \textit{half-handle} attachment. We refer to~\cite{borodzik2016morse} for definitions and further information about cobordisms of compact manifolds with boundary. Analysing the proof of Theorem~\ref{disproof} one can remark that the value of $I^\sigma_k$ does not change under handle and half-handle attachments. Since by this procedure any surface $\Sigma$ can be reduced to the disc, we get $I^\sigma_k(\Sigma)=I^\sigma_k(\mathbb D^2)=2\pi k$.


\subsection*{Plan of the paper.} The paper is organized in the following way. 
In Section~\ref{analysis} we collect all the analytic facts which are necessary for the proof of Theorem~\ref{conf&conv}. The main result here is Proposition~\ref{subdomain}. In Section~\ref{appendix4} we prove Theorem~\ref{non-bound} using the techniques developed in the previous section. Section~\ref{geometry} represents the geometric part of the paper. Here we describe convergence on the moduli space of conformal structures on a surface with boundary. Section~\ref{proofconf&conv} is devoted to the proof of Theorem~\ref{conf&conv}. In Section~\ref{main theorem proof} we deduce Theorem~\ref{disproof} from Theorem~\ref{conf&conv}. Finally, Section~\ref{appendix} contains some auxiliary technical results. 

\subsection*{Acknowledgements.} The author would like to express his gratitude to Iosif Poltero-vich, Mikhail Karpukhin, Alexandre Girouard and Bruno Colbois for stimulating discussions and useful remarks during the preparation of the paper. The author is also thankful to the reviewers for valuable remarks and helpful suggestions. This research is a part of author's PhD thesis at the Universit\'e de Montr\'eal under the supervision of Iosif Polterovich. This work is supported by the Ministry of Science and Higher Education of the Russian Federation: agreement no. 075-03-2020-223/3 (FSSF-2020-0018).

\medskip

\section{Analytic background}\label{analysis}
Here we provide a necessary analytic background that we will use in the proof of Theorem~\ref{conf&conv} in Section~\ref{proofconf&conv}. The propositions in this section are analogs of the propositions in \cite[Section 4]{karpukhin2019friedlander}. We postpone the proof of a proposition to Section \ref{appendix2} every time when it follows the exactly same way as the proof of an analogous proposition in \cite[Section 4]{karpukhin2019friedlander}.

\subsection{Convergence of Steklov-Neumann spectrum} \label{convergence} 

We start with the following convergence result.

\begin{lemma}\label{Neumann conv}
Let $(M, g)$ be a compact Riemannian manifold with boundary. Consider a finite collection $\{ B_\epsilon(p_i) \}_{i=1}^l$ of geodesic balls of radius $\epsilon$ centred at some points $p_1,\ldots,p_l \in M$. Then the spectrum of the Steklov-Neumann problem
\begin{gather}
\label{use}
\begin{cases}
\Delta_gu=0&\text{in $M\setminus \cup^l_{i=1}B_\epsilon(p_i)$},\\
\frac{\partial u}{\partial n}=0&\text{on $\cup^l_{i=1} \partial B_\epsilon(p_i) \setminus \partial M$},\\
\frac{\partial u}{\partial n}=\lambda^N_k(M \setminus \cup^l_{i=1}B_\epsilon(p_i), g) u&\text{on $\partial M \setminus \cup^l_{i=1} \partial B_\epsilon(p_i)$}
\end{cases}
\end{gather}
converges to the Steklov spectrum of $(M,g)$ as $\epsilon \to 0$.
\end{lemma}

\begin{proof}
For the sake of simplicity we only consider the case of one ball that we denote by $B_\epsilon$ centred at $p \in M$. First we consider the case when $B_\epsilon \cap \partial M \neq \varnothing$, i.e. $p\in \partial M$. 

Let $\mathcal E(u)$ denote the extension of the function $u$ by the unique solution of the problem 
\begin{gather*}
\begin{cases}
\Delta_g\mathcal E(u)=0&\text{in $B_\epsilon$},\\
\frac{\partial \mathcal E(u)}{\partial n}=0&\text{on $\partial M \cap \partial B_\epsilon$},\\
\mathcal E(u)= u&\text{on $\partial B_\epsilon \setminus \partial M$}.
\end{cases}
\end{gather*}

{\bf Claim 1.} The operator $\mathcal E(u)$ is uniformly bounded.

\begin{proof} The proof is similar to the proof of uniform boundedness of the harmonic continuation operator into small geodesic balls \cite[Example 1]{rauch1975potential}.
Fix $0<r<\epsilon$ and let $B_r$ denote a geodesic ball of radius $r$ with the same center as $B_\epsilon$. One has
\begin{gather}\label{first}
||\mathcal E(u)||^2_{L^2(B_r,g)} \leq C||u||^2_{L^2(M\setminus B_r,g)}+C||\nabla u||^2_{L^2(M\setminus B_r,g)}
\end{gather}
and
\begin{gather}\label{second}
||\nabla \mathcal E(u)||^2_{L^2(B_r,g)} \leq C||\nabla u||^2_{L^2(M\setminus B_r,g)}.
\end{gather}
Inequality~\eqref{first} follows from estimate \eqref{finally} and the trace inequality
$$
||\mathcal E(u)||^2_{L^2(B_r,g)} \leq ||\mathcal E(u)||^2_{H^1(B_r,g)} \leq C||u||^2_{H^{1/2}(\partial B_r \setminus \partial M,g)} \leq C||u||^2_{H^1(M\setminus B_r,g)}.
$$
Suppose that inequality~\eqref{second} was false. Then there exists a sequence of functions $\{u_n\}$ in $H^1(M\setminus B_r,g)$ such that
$$
||\nabla u_n||_{L^2(M\setminus B_r,g)} \leq 1/n
$$
and
$$
||\mathcal E(u_n)||_{L^2(B_r,g)} \geq 1.
$$
Consider $\alpha_n=\frac{1}{Vol(M\setminus B_r,g)}\int_{M\setminus B_r}u_ndv_g$. We show that
$$
||u_n-\alpha_n||_{H^1(M\setminus B_r,g)} \leq C/n.
$$
Indeed, by the generalized Poincar\'e inequality one has
$$
||u_n-\alpha_n||_{L^2(M\setminus B_r,g)} \leq C||\nabla u_n||_{L^2(M\setminus B_r,g)} \leq C/n
$$
moreover 
$$
||\nabla (u_n-\alpha_n)||_{L^2(M\setminus B_r,g)}=||\nabla u_n||_{L^2(M\setminus B_r,g)} \leq 1/n.
$$
Note that $\mathcal E(u_n-\alpha_n)=\mathcal E(u_n)-\alpha_n$. Then we can prove inequality~\eqref{second}
\begin{gather*}
||\nabla \mathcal E(u_n)||_{L^2(B_r,g)}=||\nabla \mathcal E(u_n-\alpha_n)||_{L^2(B_r,g)} \leq || \mathcal E(u_n-\alpha_n)||_{H^1(B_r,g)} \leq \\ \leq ||u_n-\alpha_n||_{H^{1/2}(\partial B_r \setminus \partial M,g)} \leq C||u_n-\alpha_n||_{H^{1}(M\setminus B_r,g)} \leq C/n,
\end{gather*}
where in the second and third inequalities we have used in order estimate \eqref{finally} and the trace inequality. We got a contradiction. Hence inequality~\eqref{second} is true. 

Note that for any $\rho r<\epsilon$ the first inequality scales as
$$
||\mathcal E(u)||^2_{L^2(B_{\rho r},g)} \leq C||u||^2_{L^2(M\setminus B_{\rho r},g)}+C\rho^2||\nabla u||^2_{L^2(M\setminus B_{\rho r},g)},
$$
while the second inequality scales as
$$
||\nabla \mathcal E(u)||^2_{L^2(B_{\rho r},g)} \leq C||\nabla u||^2_{L^2(M\setminus B_{\rho r},g)}.
$$
Therefore, $||\mathcal E(u)||^2_{H^1(B_{\rho r},g)} \leq C||u||^2_{L^2(M\setminus B_{\rho r},g)}+C||\nabla u||^2_{L^2(M\setminus B_{\rho r},g)}$ for $\epsilon$ small enough. 
\end{proof}

{\bf Claim 2.} One has
$$
\limsup_{\epsilon \to 0}\sigma^N_k(M\setminus B_\epsilon,g) \leq \sigma_k(M,g).
$$
\begin{proof}
We only consider the case of $B_\epsilon \cap \partial M\neq \o$. The case of $B_\epsilon \cap \partial M=\o$ is easier and follows the exactly same arguments. The proof is similar to the proof of \cite[Theorem 3.5]{bogosel2017steklov}.

Let $V_k$ be a $k-$dimensional subspace of $H^1(M,g)$ and $v\in V_k$ such that
$$
\sigma_k(M,g)=\max_{u\in V_k\setminus \{0\}}\frac{\int_M|\nabla u|^2dv_g}{\int_{\partial M}u^2ds_g}.
$$
Let $u_1,\ldots,u_k$ be an orthonormal basis in $V_k$. We modify the functions $u_i, i=1,\ldots,k$ as
$$
u_{i,\epsilon}=u_i-\frac{1}{L(\partial M \setminus \partial B_\epsilon)}\int_{\partial M\setminus \partial B_\epsilon}u_ids_g.
$$
Then $\int_{\partial M\setminus \partial B_\epsilon}u_{i,\epsilon}ds_g=0$. Consider the space $V_{k,\epsilon}:=span(u_{1,\epsilon},\ldots,u_{k,\epsilon})$. Since $\dim V_{k,\epsilon}=k$ one has
$$
\sigma^N_k(M\setminus B_\epsilon,g) \leq \max_{u_\epsilon \in V_{k,\epsilon}\setminus\{0\}}\frac{\int_{M\setminus B_\epsilon}|\nabla u_\epsilon|^2dv_g}{\int_{\partial M\setminus \partial B_\epsilon}u_\epsilon^2ds_g}.
$$
Moreover, since the dimension of $V_{k,\epsilon}$ is finite then there exists a function $v_\epsilon\in V_{k,\epsilon}$ such that
\begin{gather}\label{forclaim}
\sigma^N_k(M\setminus B_\epsilon,g) \leq \frac{\int_{M\setminus B_\epsilon}|\nabla v_\epsilon|^2dv_g}{\int_{\partial M\setminus \partial B_\epsilon}v_\epsilon^2ds_g}.
\end{gather}
Let $v_\epsilon=\sum^k_{i=1}c_iu_{i,\epsilon}$. We build the following function $v=\sum^k_{i=1}c_iu_{i} \in V_k\subset H^1(M,g)$. Note that $\nabla v_\epsilon=\sum^k_{i=1}c_i\nabla u_{i,\epsilon}=\sum^k_{i=1}c_i\nabla u_{i}=\nabla v$ on $M\setminus B_\epsilon$. Thus $\int_{M\setminus B_\epsilon}|\nabla v_\epsilon|^2dv_g=\int_{M\setminus B_\epsilon}|\nabla v|^2dv_g\to \int_{M}|\nabla v|^2dv_g$ as $\epsilon\to 0$. Moreover, it is easy to see that
\begin{gather*}
\int_{\partial M\setminus \partial B_\epsilon}v_\epsilon^2ds_g=\sum_ic^2_i\Big(\int_{\partial M\setminus\partial B_\epsilon}u^2_idv_g-\frac{1}{L(\partial M\setminus \partial B_\epsilon,g)}\Big(\int_{\partial M\setminus \partial B_\epsilon}u_ids_g\Big)^2\Big)+\\+\sum_{i\neq j}2c_ic_j\Big(\int_{\partial M\setminus \partial B_\epsilon}u_iu_jds_g-\frac{1}{L(\partial M\setminus \partial B_\epsilon,g)}\int_{\partial M\setminus \partial B_\epsilon}u_ids_g\int_{\partial M\setminus \partial B_\epsilon}u_jds_g\Big),
\end{gather*}
which converges to $\int_{\partial M}v^2ds_g$ as $\epsilon \to 0$. Then \eqref{forclaim} implies
$$
\limsup_{\epsilon \to 0}\sigma^N_k(M\setminus B_\epsilon,g) \leq \limsup_{\epsilon \to 0} \frac{\int_{M\setminus B_\epsilon}|\nabla v_\epsilon|^2dv_g}{\int_{\partial M\setminus \partial B_\epsilon}v_\epsilon^2ds_g}= \frac{\int_{M}|\nabla v|^2dv_g}{\int_{\partial M}v^2ds_g} \leq \sigma_k(M,g).
$$
\end{proof}

Now we are ready to prove the Lemma. The proof is similar to the proof of \cite[Lemma 3.2]{matthiesen2020sharp}. Let $u_\epsilon$ be a normalized $\sigma^N_k-$eigenfunction. By Claim 2 $u_\epsilon$ are uniformly bounded. If $B_\epsilon \cap \partial M=\O$ then we take the harmonic continuation into $B_\epsilon$. It is known that the operators of harmonic continuation into $B_\epsilon$ are uniformly bounded (see \cite[Example 1]{rauch1975potential}). Otherwise we extend $u_\epsilon$ into $B_\epsilon$ by $\mathcal E(u_\epsilon)$. By Claim 1 these operators are also uniformly bounded. Therefore, we get a uniformly bounded in $H^1(M,g)$ sequence $\{\tilde u_\epsilon\}$. Then there exists $\epsilon_l\to 0$ such that $\tilde u_{\epsilon_l} \rightharpoonup u$ in $H^1(M,g)$. Thus, $\tilde u_{\epsilon_l} \to u$ in $L^2(M,g)$ by the Rellich-Kondrachov embedding theorem. The standard elliptic estimates imply $u_{\epsilon_l} \to u$ in $C^\infty_{loc}(M\setminus\{p\})$. Consider a function $\varphi \in C^\infty_c(M\setminus\{p\})$ such that $supp(\varphi) \subset M\setminus B_R$ for a ball $B_R$ centred at $p$ with $R$ fixed. Extracting a subsequence by Claim 2 one can assume that $\sigma^N_k(M\setminus B_{\epsilon_l},g)\to \sigma$. Then we have
\begin{gather*}
\int_M\langle \nabla u, \nabla \varphi \rangle dv_g= \lim_{l\to 0}\int_{M\setminus B_R}\langle \nabla u_{\epsilon_l}, \nabla \varphi \rangle dv_g=\\=\lim_{l\to 0}\sigma^N_k(M\setminus B_{\epsilon_l},g)\int_{M\setminus B_R} u_{\epsilon_l} \varphi dv_g=\sigma\int_Mu\varphi dv_g.
\end{gather*}
Hence $u$ is an eigenfunction with eigenvalue $\sigma$. Thus all accumulation points of $\{\sigma^N_k(M\setminus B_{\epsilon_l},g)\}$ are in the Steklov spectrum of $M$. Our aim now is to show that $\sigma=\sigma_k(M,g)$. We will do this by showing that the $u$ is orthogonal in $L^2(\partial M,g)$ to the first $k-1$ Steklov eigenfunctions of $(M,g)$. We use the proof by induction.

Let $u_\epsilon$ be a first Steklov-Neumann eigenfunction of $(M\setminus B_\epsilon,g)$. We have already shown that $\tilde u_\epsilon \rightharpoonup u$ in $H^1(M,g)$ then by the trace embedding theorem one has $\tilde u_\epsilon \to u$ in $H^{1/2}(\partial M,g)$ and hence in $L^2(\partial M,g)$. In particular, one has $||u_\epsilon-u||_{L^2(\partial M\setminus\partial B_\epsilon,g)}\to 0$ as $\epsilon\to 0$. Then 
\begin{gather*}
|\int_{\partial M\setminus\partial B_\epsilon}(u_\epsilon -u)ds_g| \leq \int_{\partial M\setminus\partial B_\epsilon}|u_\epsilon -u|ds_g \leq \\ \leq L(\partial M\setminus \partial B_\epsilon,g)^{1/2}||u_\epsilon-u||_{L^2(\partial M\setminus\partial B_\epsilon,g)}^{1/2},
\end{gather*}
which converges tp $0$ as $\epsilon \to 0$. Since $\int_{\partial M\setminus\partial B_\epsilon}u_\epsilon ds_g=0$ one then has that $$\lim_{\epsilon\to 0}\int_{\partial M\setminus\partial B_\epsilon}uds_g=\int_{\partial M}uds_g=0.$$ Therefore, $u$ cannot be a constant and since by claim 2 $\limsup_{\epsilon\to 0}\sigma^N_1(M\setminus B_\epsilon,g)=\sigma \leq \sigma_1(M,g)$ and $\sigma$ belongs to the Steklov spectrum of $(M,g)$ we conclude that $u$ is a first Steklov eigenfunction of $(M,g)$ and $\sigma=\sigma_1(M,g)$.

Now suppose that $\limsup_{\epsilon\to 0}\sigma^N_i(M\setminus B_\epsilon,g)= \sigma_i(M,g)$ for any $i<k$. Let $u_\epsilon$ be a $k-$th Steklov-Neumann eigenfucntion of $(M\setminus B_\epsilon,g)$. Since $\tilde u_\epsilon \rightharpoonup u$ in $H^1(M,g)$ then the trace embedding theorem implies that $\tilde u_\epsilon \to u$ in $H^{1/2}(\partial M,g)$ in particular $\tilde u_\epsilon \to u$ in $L^{2}(\partial M,g)$ whence $||u_\epsilon-u||_{L^2(\partial M\setminus \partial B_\epsilon,g)}\to 0$. Let $v_\epsilon$ be an $i-$th Steklov-Neumann eigenfunction of $(M\setminus B_\epsilon,g)$ with $i<k$. Then $\int_{\partial M\setminus \partial B_\epsilon}u_\epsilon v_\epsilon ds_g=0$ moreover we have supposed that $v$ is an $i-$th Steklov eigenfunction of $(M,g)$. One has
\begin{gather*}
|\int_{\partial M\setminus\partial B_\epsilon}(u_\epsilon v_\epsilon -uv)ds_g| \leq \\ \leq \int_{\partial M\setminus\partial B_\epsilon}|u_\epsilon v_\epsilon -uv|ds_g= \int_{\partial M\setminus\partial B_\epsilon}|u_\epsilon v_\epsilon -u_\epsilon v+u_\epsilon v-uv|ds_g \leq \\ \leq \int_{\partial M\setminus\partial B_\epsilon}|u_\epsilon (v_\epsilon -v)|ds_g+\int_{\partial M\setminus\partial B_\epsilon}|v (u_\epsilon -u)|ds_g \leq \\ \leq \Big(\int_{\partial M\setminus\partial B_\epsilon}u^2_\epsilon ds_g \Big)^{1/2}\Big(\int_{\partial M\setminus\partial B_\epsilon}(v_\epsilon-v)^2 ds_g \Big)^{1/2}+\\+\Big(\int_{\partial M\setminus\partial B_\epsilon}v^2_\epsilon ds_g \Big)^{1/2}\Big(\int_{\partial M\setminus\partial B_\epsilon}(u_\epsilon-u)^2 ds_g \Big)^{1/2}\to 0~\text{as $\epsilon \to 0$}.
\end{gather*}
Hence $\int_{\partial M\setminus\partial B_\epsilon}u_\epsilon v_\epsilon ds_g\to \int_{\partial M}u v ds_g$ as $\epsilon \to 0$. But $\int_{\partial M\setminus\partial B_\epsilon}u_\epsilon v_\epsilon ds_g=0$ for all $\epsilon$. Thus $\int_{\partial M}u v ds_g=0$. We conclude that $u$ is orthogonal in $L^2(\partial M,g)$ to the first $k-1$ Steklov eigenfunctions. Thus $\sigma=\sigma^N_k(M,g)$.

\end{proof}

We endow the set of Riemannian metrics on $\Sigma$ with the $C^\infty-$topology. Then the following "continuity" result holds. 

\begin{proposition}\label{N-cont}
Let $\Sigma$ be a surface with boundary and $\Omega\subset \Sigma$ be a Lipschitz domain. Let the sequence of Riemannian metrics $g_m$ on $\Sigma$ converge in $C^\infty-$topology to the metric $g$. Then $\sigma^*_k(\Sigma,[g_m])\to\sigma^*_k(\Sigma,[g])$. Similarly, if $h_m|_{\overline\Omega}$ converge to $g|_{\overline\Omega}$ in $C^\infty$-topology, then $\sigma^{N*}_k(\Omega, \partial^S\Omega, [h_m|_{\overline\Omega}])\to\sigma^{N*}_k(\Omega,\partial^S\Omega, [g|_{\overline\Omega}])$. 
\end{proposition}
\begin{proof}
We provide a proof for the functional $\sigma^*_k(\Sigma,[g])$. The proof for the functional $\sigma^{N*}_k(\Omega,[g|_{\overline\Omega}])$ follows the exactly same arguments.

Choose any $\varepsilon>0$ and consider $m$ large enough. One has 
$$
\frac{1}{1+\varepsilon} f g_m(v,v) \leq f g(v,v) \leq (1+\varepsilon) f g_m(v,v),\quad \forall v \in \Gamma(TM\setminus\{0\}),
$$
where $f$ is any positive smooth function on $\Sigma$.
Then by \cite[Proposition 32]{colbois2016steklov} one has
$$
\frac{1}{(1+\varepsilon)^{6}}\bar \sigma_k(\Sigma,fg_m) \leq\bar \sigma_k(\Sigma,fg) \leq (1+\varepsilon)^{6}\bar \sigma_k(\Sigma,fg_m).
$$
Taking the supremum over all $f$ yields
$$
\frac{1}{(1+\varepsilon)^{6}}\sigma^*_k(\Sigma,[g_m]) \leq\sigma^*_k(\Sigma,[g]) \leq (1+\varepsilon)^{6}\sigma^*_k(\Sigma,[g_m]),
$$
which completes the proof since this inequality holds for any $\varepsilon>0$.
\end{proof}
\subsection{Discontinuous metrics}\label{main lemma proof}

Let $\Sigma$ be a compact surface with boundary. Consider a set of pairwise disjoint Lipschitz domains $\{\Omega_i\}^s_{i=1}$ in $\Sigma$ such that $\Sigma=\bigcup^s_{i=1} \overline\Omega_i$. Let $C^{\infty}_+(\Sigma,\{\Omega_i\})$ denote a set of functions on $\bigcup^s_{i=1} \overline\Omega_i$ such that $\rho\in C^{\infty}_+(\Sigma,\{\Omega_i\})$ means that $\rho|_{\Omega_i} = \rho_i\in C^\infty(\overline\Omega_i)$ are positive for every $i$. Similarly, $C^{\infty}(\Sigma,\{\Omega_i\})$ denotes a set of ''smooth'' functions on $\bigcup^s_{i=1} \overline\Omega_i$. We introduce discontinuous metrics on $\Sigma$ defined as $\rho g\in[g]$, where $\rho\in C^{\infty}_+(\Sigma,\{\Omega_i\})$ and $g$ is a genuine Riemannian metric. The set $C^{k}(\Sigma,\{\Omega_i\})$ of functions which are of class $C^k$ in every $\overline\Omega_i$ is defined in a similar way. The Steklov spectrum of the metric $\rho g$ is defined as the set of critical values of the Rayleigh quotient
$$
R_{\rho g}[\varphi]=\frac{\displaystyle\int_\Sigma | \nabla_g \varphi|^2_g dv_g}{\displaystyle\int_{\partial \Sigma} \rho^{\frac{1}{2}}  \varphi^2 ds_g}.
$$ 
This is the Rayleigh quotient of the \textit{Steklov problem with density $\rho$.} The Steklov spectrum with density $\rho$ is well-defined for any non-negative $\rho \in L^\infty(\Sigma,g)$ (see~\cite[Proposition 1.3]{kokarev2014variational}). Elliptic regularity implies that the eigenfunctions are at least $1/2-$H\"older continuous on $\partial\Sigma$. Therefore, Steklov eigenvalues of the metric $\rho g$ admit the following variational characterization
$$
\sigma_k(\Sigma,\rho g) = \inf_{E_{k+1}} \sup_{\varphi\in E_{k+1}} R_{\rho g}[\varphi],
$$
where $E_{k+1}$ ranges over all $(k+1)$-dimensional subspaces of $C^0(\Sigma)$.


We introduce the following notation
\begin{align*}
\sigma^*_k(\Sigma,\{\Omega_i\},[g])=\sup \{ \bar\sigma_k(\rho g)~\vert~ \rho \in  C^{\infty}_+(\Sigma,\{\Omega_i\})\},
\end{align*}
where $\bar\sigma_k(\rho g)$ is the normalized $k$-th eigenvalue given by
$$
\bar\sigma_k(\rho g) = \sigma_k(\rho g) L_{\rho g}(\partial \Sigma).
$$

The following lemma particularly asserts that the quantity $\sigma^*_k(\Sigma,\{\Omega_i\},[g])$ is well-defined.

\begin{lemma}\label{identity}
Let $(\Sigma,g)$ be a Riemannian surface with boundary. Consider a set of pairwise disjoint Lipschitz domains $\Omega_i$ such that $\Sigma=\bigcup^s_{i=1} \overline\Omega_i$. Then one has
\begin{align*}
\sigma^*_k(\Sigma,\{\Omega_i\},[g])=\sigma^*_k(\Sigma,[g])
\end{align*}
\end{lemma} 
\begin{proof}
The proof follows the same steps as the proof of Lemma 2 in the paper \cite{MR1717641}. We provide it here.

Since the set of discontinuous metrics is larger than the set of continuous ones, we have $\sigma^*_k(\Sigma,\{\Omega_i\},[g])) \geq \sigma^*_k(\Sigma,[g])$. Therefore, we have to prove that  
\begin{gather*}
\sigma^*_k(\Sigma,\{\Omega_i\},[g])) \leq \sigma^*_k(\Sigma,[g]),
\end{gather*}
which is equivalent to
\begin{align}\label{lambda_inequality}
\sigma_k(\Sigma,\rho g) \leq \sigma^*_k(\Sigma,[g]), 
\end{align}
where $\rho \in C^{\infty}_{+}(\Sigma,\{\Omega_i\})$ and $ \int_{\partial\Sigma} \rho^{1/2}ds_g =1$. 



Let $E_k$ be the eigenspace corresponding to the $k$-th Steklov eigenvalue of the metric $\rho g$. We put

\begin{align*}
S=\{u\in H^1(\Sigma,\rho g)~|~u \perp_{L^2(\partial\Sigma, \rho g)} E_0,\dots,E_{k-1}, \int_{\partial\Sigma} \rho^{1/2}u^2 ds_g =1\}
\end{align*}
For any $\varepsilon>0$ we consider the functional
\begin{align*}
\mathcal{F}_\rho [u]:=\int_{\Sigma}|\nabla_gu|^2dv_g-(\sigma_k(\Sigma,\rho g)-\varepsilon)\int_{\partial\Sigma}\rho^{1/2}u^2ds_g.
\end{align*}
It immediately follows that $\mathcal{F}_\rho[u] \geq\varepsilon, \forall u\in S$. 

Let $0<a:=\min_{\cup\{\Omega_i\}}\rho$ and $\max_{\cup\{\Omega_i\}}=:b<\infty$. We define a smooth non-decreasing function $\chi(t)$ on $\mathbb{R}_+$ that equals zero if $t<1/2$ and equals 1 when $t>1$ and define the following parametrized family of functions:
\begin{equation*}
\rho_\delta(x) = 
 \begin{cases}
   \rho(x) &\text{if $x \notin U$}\\
   \rho(x)\chi \Big(\frac{d^2(x)}{\delta}\Big)+b\Big(1-\chi \Big(\frac{d^2(x)}{\delta}\Big)\Big) &\text{if $x \in U$}
 \end{cases}
\end{equation*}
where $d$ is the distance function from a point $x\in \Sigma$ to $\cup\{\partial\Omega_i\cap \partial \Omega_j\},~i\neq j$ and $U$ is a sufficiently small tubular neighborhood of $\cup\{\partial\Omega_i\cap \partial \Omega_j\},~i\neq j$ where $d^2$ is smooth. We have:

\begin{enumerate}[(i)]
\item $\Big(\frac{a}{b}\Big)\rho \leq \rho_\delta  \leq \Big(\frac{b}{a}\Big) \rho$;
\item $\lim_{\delta \to 0} \int_{\partial\Sigma} \rho^{1/2}_\delta  ds_g=1$;
\item $\lim_{\delta \to 0} \int_{\partial\Sigma} |\rho^{1/2}_\delta - \rho^{1/2}|^q ds_g=0, \forall q<\infty$.
\end{enumerate} 
We want to prove that $\mathcal{F}_{\rho_\delta}[u] \geq 0, \forall u \in S$.

Consider $T=(\sigma_k(\Sigma,\rho g) -\varepsilon)\sqrt\frac{b}{a}$ and divide the set $S$ into two parts $S_1$ and $S_2$:

\begin{gather*}
S_1:=\{u \in S | \int_{\Sigma} |\nabla_g u|^2 dv_g \geq T\}, \\
S_2:=S \setminus S_1= \{u \in S | \int_{\Sigma} |\nabla_g u|^2 dv_g < T\}.
\end{gather*} 

If $u \in S_1$ then
\begin{gather*}
\mathcal{F}_{\rho_\delta}[u]=\int_{\Sigma}|\nabla_gu|^2dv_g-(\sigma_k(\Sigma,\rho g)-\varepsilon)\int_{\partial\Sigma}\rho_\delta^{1/2}u^2ds_g \geq\\ \geq (\sigma_k(\Sigma,\rho g)-\varepsilon)\Big(\sqrt\frac{b}{a}-\int_{\partial\Sigma}\rho_\delta^{1/2}u^2ds_g\Big) \geq \\ \geq(\sigma_k(\Sigma,\rho g)-\varepsilon)\sqrt\frac{b}{a}(1-\int_{\partial\Sigma}\rho^{1/2}u^2ds_g)=0.
\end{gather*}

Let us show that $||u||_{L^p(\partial \Sigma,g)}$ with $p\geq 2$ is bounded for any $u \in S_2$. We consider the following operator $L[u]:=\int_{\partial \Sigma}u\rho^{1/2}ds_g$. For this operator one has
 $$
 |L[u]|\leq C\int_{\partial\Sigma}|u|ds_g\leq C||u||_{L^2(\partial\Sigma,g)}\leq C||u||_{H^1(\Sigma,g)},
 $$
which implies that $L\in H^{-1}(\Sigma,g)$. Here we used in order the boundedness of $\rho$, the Cauchy-Schwarz and the trace inequalities. We also used the convention that $C$ denotes any positive constant depending only on $\Sigma$. \cite[Lemma 8.3.1]{MR1411441} applied to the operator $L$ implies that there exists a constant $C>0$ depending only on $\Sigma$ such that
$$
||u||^2_{L^2(\Sigma,g)}\leq C||\nabla u||^2_{L^2(\Sigma,g)}<CT,
$$
where we used the fact that $L[u]=0~\forall u\in S$. By the trace theorem one then has
$$
||u||^2_{H^{1/2}(\partial\Sigma,g)}\leq C'||u||^2_{H^{1}(\Sigma,g)}<C'',
$$
where $C''=C'(CT+T)$. Finally by the Sobolev embedding theorem (see for instance \cite[Theorem 6.9]{MR2944369}) we get
$$
||u||_{L^p(\partial \Sigma,g)}\leq C'''||u||_{H^{1/2}(\partial\Sigma,g)}<\tilde C~\forall 2\leq p<\infty,
$$
where $\tilde C=C'''\sqrt{C''}$. 
Therefore, if $u \in S_2$ then
\begin{gather*}
\mathcal{F}_{\rho_\delta}[u]=\int_{\Sigma}|\nabla_gu|^2dv_g-(\sigma_k(\Sigma,\rho g)-\varepsilon)\int_{\partial\Sigma}\rho_\delta^{1/2}u^2ds_g =\\=\int_{\Sigma}|\nabla_gu|^2dv_g- (\sigma_k(\Sigma,\rho g)-\varepsilon)-(\sigma_k(\Sigma,\rho g)-\varepsilon)\int_{\partial\Sigma}(\rho_\delta^{1/2}-\rho^{1/2})u^2ds_g \geq \\ \geq \varepsilon-(\sigma_k(\Sigma,\rho g)-\varepsilon)\Big(\int_{\partial\Sigma}(\rho_\delta^{1/2}-\rho^{1/2})^qds_g\Big)^{1/q}\Big(\int_{\partial\Sigma}|u|^pds_g\Big)^{2/p} \geq\\ \geq \varepsilon-(\sigma_k(\Sigma,\rho g)-\varepsilon)\frac{\varepsilon}{\sigma_k(\Sigma,\rho g)-\varepsilon}=0.
\end{gather*}
In the last inequality we put 
$$\Big(\int_{\partial\Sigma}(\rho_\delta^{1/2}-\rho^{1/2})^qds_g\Big)^{1/q}\Big(\int_{\partial\Sigma}|u|^pds_g\Big)^{2/p}=\frac{\varepsilon}{\sigma_k(\Sigma,\rho g)-\varepsilon}
$$
 since $\int_{\partial\Sigma}(\rho_\delta^{1/2}-\rho^{1/2})^qds_g \to 0$ as $\delta \to 0$ and $\int_{\partial\Sigma}|u|^pds_g<+\infty$. 


Hence, $\mathcal{F}_{\rho_\delta}[u] \geq 0, \forall u \in S$ which implies $\sigma_k(\Sigma,\rho_\delta g) \geq \sigma_k (\Sigma, \rho g) -\varepsilon$. We then have
\begin{gather*}
\bar\sigma_k(\Sigma,\rho_\delta g)=\sigma_k(\Sigma,\rho_\delta g)L_{\rho_\delta g}(\partial \Sigma)  \geq \sigma_k (\Sigma, \rho g) L_{\rho_\delta g}(\partial \Sigma) -\varepsilon L_{\rho_\delta g}(\partial \Sigma).
\end{gather*}



Therefore, $\sigma^*_k (\Sigma, [g]) \geq \sigma_k (\Sigma, \rho g) L_{\rho_\delta g}(\partial \Sigma) -\varepsilon L_{\rho_\delta g}(\partial \Sigma)$. Letting $\delta \to 0$ one then gets  $\sigma^*_k (\Sigma, [g]) \geq \sigma_k (\Sigma, \rho g)-\varepsilon$ that implies \eqref{lambda_inequality} since $\varepsilon$ is arbitrary small.
\end{proof}

Lemma \ref{identity} implies the following lemma whose proof is postponed to Section \ref{appendix2}.

\begin{lemma}
\label{identity2}
Let $(\Sigma,g)$ be a Riemannian surface with boundary. Consider a set of pairwise disjoint domains $\Omega_i$ such that $\Sigma=\bigcup^s_{i=1} \overline\Omega_i$ and $\Omega_i\cap \partial\Sigma=\partial^S\Omega_i$. Let $(\Omega,h) = \sqcup(\overline\Omega_i,g|_{\overline\Omega_i})$ and $\partial^S\Omega=\sqcup \partial^S\Omega_i$. Then for all $k \geq 0$ one has
$$
\sigma^*_k(\Sigma,[g]) \geq \sigma^{N*}_k(\Omega,\partial^S\Omega, [h]).
$$
\end{lemma}

\subsection{Steklov-Neumann spectrum of a subdomain.} 
This section is devoted to the following technical lemma 
\begin{lemma}\label{liminf}
 Let $\rho_\delta \in C^\infty_+(\Sigma,\{\Omega,\Sigma\setminus\Omega\})$ such that $\rho_\delta|_{\Omega}\equiv 1$ and $\rho_\delta|_{\Sigma\setminus\Omega}\equiv\delta$. Then one has $$\liminf_{\delta \to 0}\sigma_k(\rho_\delta g) \geq \sigma^{N}_k(\Omega,\partial^S\Omega, g),$$ where $\sigma^{N*}_k(\Omega,\partial^S\Omega, g)$ is the $k$-th Steklov Neumann eigenvalue of the domain $(\Omega,g)$ and $\partial^S\Omega=\partial\Sigma\cap\Omega \neq\O$. 
\end{lemma}

\begin{proof}
The idea of the proof comes from the proof of~\cite[Section 2, Step 2]{enciso2015}.

{\bf Case I.} First we consider the case when $\Omega^c \cap \partial \Sigma \neq \O$. Let $\Omega^c$ denotes $int(\Sigma\setminus\Omega)$ and $\partial^S\Omega^c=\partial \Omega^c \cap \partial \Sigma$. Since by elliptic regularity eigenfunctions of the Steklov problem with bounded density are in $H^1$ on the boundary we can restrict ourselves to the space $H^1(\partial\Sigma,g)$. More precisely, let $\psi$ be an eigenfunction with eigenvalue $\sigma$ then by elliptic regularity:
$$
||\psi||^2_{H^1(\partial\Sigma,\rho_\delta g)} \leq C(||\sigma\psi||^2_{L^2(\partial\Sigma,\rho_\delta g)}+||\psi||^2_{L^2(\partial\Sigma,\rho_\delta g)}) \leq C(\sigma^2+1)||\psi||^2_{L^2(\partial\Sigma,\rho_\delta g)}
$$
for some positive constant $C$. This implies
$$
\frac{||\nabla \psi||^2_{L^2(\partial\Sigma,\rho_\delta g)}}{||\psi||^2_{L^2(\partial\Sigma,\rho_\delta g)}} \leq C(\sigma^2+1)-1.
$$
More generally we see that if $\varphi \in \myspan\langle \psi_0,\ldots,\psi_k \rangle$, where $\psi_i$ is in the $i$-th eigenspace of $(\Sigma,g_\delta)$ then there exists a constant $C_k>0$ such that 
$$
\frac{||\nabla \varphi ||^2_{L^2(\partial\Sigma,\rho_\delta g)}}{|| \varphi ||^2_{L^2(\partial\Sigma,\rho_\delta g)}} \leq C_k.
$$
Therefore, we set
$$
\mathcal H:=\{\varphi \in H^1(\partial\Sigma,g)~|~\frac{||\nabla \varphi ||^2_{L^2(\partial\Sigma,\rho_\delta g)}}{|| \varphi ||^2_{L^2(\partial\Sigma,\rho_\delta g)}} \leq C_k\},
$$
$$
\mathcal{H}_1:=\{\varphi \in \mathcal H ~|~\frac{\partial\hat\varphi}{\partial n}=0~\text{on $\partial^S\Omega^c$}\},
$$
where $\hat \varphi$ stands for the harmonic continuation of $\varphi$ into $\Sigma$ and
$$
\mathcal{H}_2:=\{\varphi \in \mathcal H ~|~\varphi \in H^1_0(\partial^S\Omega^c,g),~\varphi_{|_\Omega}=0\}.
$$
{\bf Claim 1.} One has 
$$
\int_\Sigma\langle\nabla \hat\varphi, \nabla \hat\psi \rangle_{\tilde g} dv_{\tilde g}=0, \forall \varphi \in \mathcal{H}_1, \psi \in \mathcal{H}_2,
$$
for any metric $\widetilde g\in[g]$.
\begin{proof} 
\begin{gather*}
\int_\Sigma\langle\nabla \hat\varphi, \nabla \hat\psi \rangle_{\tilde g} dv_{\tilde g}=\int_\Sigma \Delta_{\tilde g} \hat\varphi \hat\psi dv_{\tilde g}+\int_{\partial \Sigma}\frac{\partial\hat\varphi}{\partial \tilde n}\psi ds_{\tilde g}=\\=\int_{\partial^S\Omega^c}\frac{\partial\hat\varphi}{\partial \tilde n}\psi ds_{\tilde g}+\int_{\partial^S\Omega}\frac{\partial\hat\varphi}{\partial \tilde n}\psi ds_{\tilde g}=0.
\end{gather*}
\end{proof}

 For the sake of simplicity we use the symbols $\sigma^\delta_k$ for $\sigma_k(\rho_\delta g)$, $g_\delta$ for $\rho_\delta g$ and $R_\delta$ for the Rayleigh quotient 
$$
R_\delta[\varphi]=\frac{\displaystyle\int_\Sigma|\nabla \hat\varphi|^2_{g_\delta}dv_{g_\delta}}{\displaystyle\int_{\partial \Sigma} \varphi^2 ds_{g_\delta}}.
$$


{\bf Claim 2.} \label{bound}
There exists a constant that we also denote by $C_k>0$ such that $\sigma^\delta_k \leq C_k$.

\begin{proof}
Theorem \ref{Kor} implies that there exists a constant $C(k)>0$ such that
$$
\sigma^*_k(\Sigma, [g]) \leq C(k).
$$
By Lemma~\ref{identity} for every $\delta$ one has
$$
\sigma^\delta_k L_{g_\delta}(\partial \Sigma) \leq \sigma^*_k(\Sigma, [g]) \leq C(k).
$$
Therefore,
$$
\sigma^\delta_k \leq \frac{C(k)}{L_{g_\delta}(\partial \Sigma)}=\frac{C(k)}{L_{g}(\partial^S\Omega)+ \delta ^{1/2}L_{g}(\partial^S \Omega^c)} \leq \frac{C(k)}{L_{g}(\partial^S \Omega)}=C_k.
$$
\end{proof}

Let $W_k$ be the set of $k+1$-dimensional subspaces of $\mathcal H$ satisfying the condition that ${R_\delta}|_{W_k} \leq C_k$. Claim 2 particularly implies that the space spanned by the first $k+1$ eigenfunctions is in $W_k$, i.e. $W_k\ne\O$.

Consider the operator $\mathcal E$ defined in section \ref{convergence} by 
\begin{gather*}
\begin{cases}
\Delta_{g} \mathcal E(u)=0&\text{in $\Sigma$},\\
\frac{\partial \mathcal E(u)}{\partial n}=0&\text{on $\partial^S\Omega^c$},\\
\mathcal E(u)=u&\text{on $\partial^S\Omega$}.
\end{cases}
\end{gather*}

 For a function $\varphi \in H^1(\partial \Sigma,g)$ we fix its decomposition $\varphi_1+\varphi_2$ with
\begin{align*}
\varphi_1=
 \begin{cases}
 \varphi&\text{on $\partial^S\Omega$},\\
\mathcal E(\varphi)&\text{on $\partial^S\Omega^c$}
 \end{cases}
\end{align*}
and $\varphi_2=\varphi_1-\varphi$. Note that $\hat \varphi_1=\mathcal E(\varphi_1).$

{\bf Claim 3.} 
For every $\varphi \in V \in W_k$ there exists a constant $C>0$ such that
$$
\displaystyle\int_{\partial^S \Omega^c}\varphi^2_2~ds_{g_\delta} \leq C\sqrt{\delta} \displaystyle\int_{\partial \Sigma} \varphi^2 dv_{g_\delta}. 
$$

\begin{proof}
By Claim 1 one has
 $$
 \int_\Sigma\langle\nabla \hat\varphi_1,\nabla \hat\varphi_2 \rangle_{g} dv_{g}=0.
 $$ 
Further, since $\varphi \in V \in W_k$ we have 
\begin{gather*}
C_k \geq R_\delta[\varphi]=\frac{\displaystyle\int_\Sigma|\nabla \hat\varphi|^2_{g}dv_{g}}{\displaystyle\int_{\partial \Sigma} \varphi^2 ds_{g_\delta}}=\frac{\displaystyle\int_\Sigma|\nabla \hat\varphi_1|^2dv_{g}+\displaystyle\int_{\Sigma}|\nabla \hat\varphi_2|^2_{g}dv_{g}}{\displaystyle\int_{\partial \Sigma} \varphi^2 ds_{g_\delta}} \geq \\ \geq \frac{\displaystyle\int_{\Omega^c}|\nabla \hat\varphi_2|^2_{g}dv_{g}}{\displaystyle\int_{\partial \Sigma} \varphi^2 ds_{g_\delta}}=\frac{1}{\delta^{1/2}}\frac{\displaystyle\int_{\Omega^c}|\nabla \hat\varphi_2|^2_{g}dv_{g}}{\displaystyle\int_{\partial^S\Omega^c} \varphi^2_2 ds_{g}}\frac{||\varphi_2||^2_{L^2(\partial^S\Omega^c, g_\delta)}}{||\varphi||^2_{L^2(\partial \Sigma, g_\delta)}} \geq \\ \geq \frac{\sigma^D_1(\Omega^c,\partial^S\Omega^c, g)}{\sqrt{\delta}} \frac{||\varphi_2||^2_{L^2(\partial^S\Omega^c, g_\delta)}}{||\varphi||^2_{L^2(\partial \Sigma, g_\delta)}},
\end{gather*}
where $\sigma^D_1(\Omega^c,\partial^S\Omega^c, g)$ is the first non-zero Steklov-Dirichlet eigenvalue of $(\Omega^c,g)$ (see \cite{banuelos2010eigenvalue}).
\end{proof}

{\bf Claim 4.} 
For every $\varphi \in V \in W_k$ and for every sufficiently small $\delta$ there exists a constant $C>0$ such that
$$
\int_{\partial \Sigma}\varphi^2~ds_{g_\delta} \leq (1+C \delta^{1/4}) \int_{\partial \Sigma} \varphi^2_1 ds_{g_\delta}. 
$$

\begin{proof}
One has
\begin{align*}
||\varphi||^2_{L^2(\partial \Sigma, g_\delta)}=\int_{\partial^S\Omega^c}(\varphi_1+\varphi_2)^2dv_{s_\delta}+\int_{\partial^S\Omega}\varphi^2_1ds_{g_\delta} \leq\\ \leq \Big(1+\frac{1}{\varepsilon}\Big) \int_{\partial \Sigma}\varphi_2^2ds_{g_\delta}+(1+\varepsilon) \int_{\partial \Sigma}\varphi^2_1ds_{g_\delta},
\end{align*}
for every $\varepsilon>0$. Applying Claim 3 we obtain
\begin{align*}
||\varphi||^2_{L^2(\partial \Sigma, g_\delta)} \leq C\sqrt{\delta}\Big(1+\frac{1}{\varepsilon}\Big) \int_{\partial \Sigma}\varphi^2ds_{g_\delta}+(1+\varepsilon) \int_{\partial \Sigma}\varphi^2_1ds_{g_\delta},
\end{align*}
and hence,
\begin{align*}
\Big(1-C\sqrt{\delta} \Big(1+\frac{1}{\varepsilon}\Big)\Big)||\varphi||^2_{L^2(\partial \Sigma, g_\delta)} \leq (1+\varepsilon) ||\varphi_1||^2_{L^2(\partial \Sigma, g_\delta)}.
\end{align*}
Choosing $\varepsilon=\delta^{1/4}$ completes the proof.  
\end{proof}

{\bf Claim 5.} \label{C3}
For every $\varphi \in V \in W_k$ and for every sufficiently small $\delta$ there exists a constant $C>0$ such that
$$
\int_{\partial^S\Omega^c}\varphi^2_1~ds_{g} \leq C\int_{\partial^S\Omega} \varphi^2_1 ds_{g}. 
$$

\begin{proof}
\begin{gather*}
C_k \geq \frac{\displaystyle\int_{\partial\Sigma}|\nabla \varphi|^2_{g_\delta}dv_{g_\delta}}{\displaystyle\int_{\partial \Sigma} \varphi^2 ds_{g_\delta}} \geq \frac{\displaystyle\int_{\partial^S\Omega}|\nabla \varphi|^2_{g}ds_{g}}{\displaystyle\int_{\partial \Sigma} \varphi^2 ds_{g_\delta}}=\frac{\displaystyle\int_{\partial^S\Omega}|\nabla \varphi_1|^2_{g}ds_{g}}{\displaystyle\int_{\partial \Sigma} \varphi^2 ds_{g_\delta}},
\end{gather*}
since $\varphi=\varphi_1$ on $\partial^S\Omega$. Then by Claim 4 one has
\begin{gather*}
C_k \geq\frac{\displaystyle\int_{\partial^S\Omega}|\nabla \varphi_1|^2_{g}ds_{g}}{\displaystyle\int_{\partial \Sigma} \varphi^2 ds_{g_\delta}} \geq \frac{1}{1+C\delta^{1/4}}\frac{\displaystyle\int_{\partial^S\Omega}|\nabla \varphi_1|^2_{g}ds_{g}}{\displaystyle\int_{\partial \Sigma} \varphi_1^2 ds_{g_\delta}},
\end{gather*}
which implies
\begin{equation}\label{thatsit}
\begin{split}
\int_{\partial^S\Omega}|\nabla \varphi_1|^2_{g}ds_{g} \leq C_k(1+C\delta^{1/4})\int_{\partial \Sigma} \varphi_1^2 ds_{g_\delta}=\\=C_k(1+C\delta^{1/4})\Big(\int_{\partial^S\Omega} \varphi_1^2 ds_{g}+\delta^{1/2}\int_{\partial^S\Omega^c} \varphi_1^2 ds_{g}\Big).
\end{split}
\end{equation}
For the rest of the proof $C$ stands for any positive constant depending possibly on $\Sigma$ and $g$ but not on $\delta$ or $\varphi$. 

Note that $\partial^s\Omega$ has positive capacity (see \cite[pp.102-105]{MR3791463}). Applying in order the trace inequality, estimate \eqref{finally}, the Sobolev embedding and inequality \eqref{thatsit} yield
\begin{gather*}
||\varphi_1||^2_{L^2(\partial^S\Omega^c,g)} \leq C||\hat \varphi_1||^2_{H^1(\Sigma,g)} \leq C||\varphi_1||^2_{H^{1/2}(\partial^S\Omega,g)} \leq \\ \leq C||\varphi_1||^2_{H^{1}(\partial^S\Omega,g)}=C(||\varphi_1||^2_{L^{2}(\partial^S\Omega,g)}+||\nabla \varphi_1||^2_{L^{2}(\partial^S\Omega,g)}) \leq \\ \leq C(1+C\delta^{1/4})\Big(||\varphi_1||^2_{L^{2}(\partial^S\Omega,g)}+\delta^{1/2}||\varphi_1||^2_{L^{2}(\partial^S\Omega^c,g)}\Big),
\end{gather*}
which implies the required inequality for $\delta$ small enough.
\end{proof}

Further by  the fact that $\int_\Sigma\langle\nabla \hat\varphi_1, \nabla \hat\varphi_2 \rangle_{g} dv_{g}=0$ and by claim 4 for every $\varphi \in V \in W_k$ and one has
\begin{gather*}
R_\delta[\varphi]=\frac{\displaystyle\int_\Sigma|\nabla \hat\varphi|^2_{g}dv_{g}}{\displaystyle\int_{\partial \Sigma} \varphi^2 ds_{g_\delta}}=\frac{\displaystyle\int_\Sigma|\nabla \hat\varphi_1|^2_{g}dv_{g}+\displaystyle\int_{\Sigma}|\nabla \hat\varphi_2|^2_{g}dv_{g}}{\displaystyle\int_{\partial \Sigma} \varphi^2 ds_{g_\delta}} \geq \\ \geq \frac{1}{1+C\delta^{1/4}} \frac{\displaystyle\int_\Sigma|\nabla \hat\varphi_1|^2_{g}dv_{g}+\int_{\Sigma}|\nabla \hat\varphi_2|^2_{g}dv_{g}}{\displaystyle\int_{\partial \Sigma} \varphi_1^2 ds_{g_\delta}} \geq \\ \geq \frac{1}{1+C\delta^{1/4}} \frac{\displaystyle\int_\Sigma|\nabla \hat\varphi_1|^2_{g}dv_{g}}{\displaystyle\int_{\partial \Sigma} \varphi_1^2 ds_{g_\delta}}=\frac{1}{1+C\delta^{1/4}} \frac{\displaystyle\int_\Sigma|\nabla \hat\varphi_1|^2_{g}dv_{g}}{\displaystyle\int_{\partial^S\Omega} \varphi_1^2 dv_{g}+\delta^{1/2}\int_{\partial^S\Omega^c}\varphi_1^2 dv_{g}} 
\end{gather*}
and by claim 5 we get
\begin{gather*}
R_\delta[\varphi] \geq\frac{1}{(1+C\delta^{1/4})(1+\delta^{1/2}C)} \frac{\displaystyle\int_\Sigma|\nabla \hat\varphi_1|^2_{g}dv_{g}}{\displaystyle\int_{\partial^S\Omega} \varphi_1^2 ds_{g}} \geq\\ \geq \frac{1}{(1+C\delta^{1/4})(1+\delta^{1/2}C)} \frac{\displaystyle\int_\Omega|\nabla \hat\varphi_1|^2_{g}dv_{g}}{\displaystyle\int_{\partial^S\Omega} \varphi_1^2 ds_{g}} \geq \\ \geq\frac{1}{(1+C\delta^{1/4})(1+\delta^{1/2}C)} R^N_{(\Omega,\partial^S\Omega, g)}[\varphi_{|_\Omega}].
\end{gather*}
where $R^N_{(\Omega,\partial^S\Omega, g)}$ denotes the Rayleigh quotient for the Steklov-Neumann problem in the domain $(\Omega,g)$. 

Let $V=\myspan\langle \psi_0,\ldots,\psi_k \rangle$, where $\psi_i$ is in the $i$-th eigenspace of $(\Sigma,g_\delta)$. Then
\begin{equation}\label{before}
\begin{split}
\sigma^\delta_k=\max_{\varphi \in V}R_\delta[\varphi] \geq \frac{1}{(1+C\delta^{1/4})(1+\delta^{1/2}C)} \max_{\varphi \in V}R^N_{(\Omega,\partial^S\Omega, g)}[\varphi_{|_\Omega}] \geq \\ \geq \frac{1}{(1+C\delta^{1/4})(1+\delta^{1/2}C)}\sigma^N_k(\Omega,\partial^S\Omega, g),
\end{split}
\end{equation}
since the restriction to $\Omega$ of the functions $\psi_i$ form the space of the same dimension by unique continuation. Finally, passing to the $\liminf$ as $\delta \to 0$ in~\eqref{before} yields the lemma.

{\bf Case II.} The case when $\Omega^c \cap \partial \Sigma=\O$ is trivial. Indeed, in this case we have $\partial^S\Omega=\partial \Sigma$. Then for any function $\varphi$ one has
\begin{gather*}
R_\delta[\varphi]=\frac{\displaystyle\int_\Sigma|\nabla \hat\varphi|^2_{g}dv_{g}}{\displaystyle\int_{\partial \Sigma} \varphi^2 ds_{g_\delta}} \geq \frac{\displaystyle\int_\Omega|\nabla \hat\varphi|^2_{g}dv_{g}}{\displaystyle\int_{\partial^S\Omega} \varphi^2 ds_{g}}=R^N_{(\Omega,\partial^S\Omega, g)}[\varphi_{|_\Omega}].
\end{gather*}
Therefore, considering $V=\myspan\langle \psi_0,\ldots,\psi_k \rangle$, where $\psi_i$ is in the $i$-th eigenspace of $(\Sigma,g_\delta)$ yields 
\begin{gather*}
\sigma^\delta_k=\max_{\varphi \in V}R_\delta[\varphi] \geq  \max_{\varphi \in V}R^N_{(\Omega,\partial^S\Omega, g)}[\varphi_{|_\Omega}] \geq \sigma^N_k(\Omega,\partial^S\Omega, g).
\end{gather*}
Taking $\liminf$ as $\delta \to 0$ completes the proof.
\end{proof}

Lemma \ref{liminf} is the key ingredient in the proof of the following proposition. We postpone the proof to Section \ref{appendix2}.

\begin{proposition}
\label{subdomain}
Let $(\Sigma,g)$ be a Riemannian surface with boundary, $\Omega\subset \Sigma$ a Lipschitz domain and $\partial^S\Omega=\partial\Sigma\cap\Omega\neq\O$. Then for all $k$ one has
$$
\sigma^*_k(\Sigma,[g]) \geq \sigma^{N*}_k(\Omega,\partial^S\Omega, [g|_{\overline\Omega}]).
$$
Similarly, let $(\Sigma,g)$ be a Riemannian surface whose boundary. Let $\partial^S\Sigma$ denote all boundary components with the Steklov boundary condition and $\Omega\subset \Sigma$ be a Lipschitz domain such that $\partial^S\Omega \subset \partial^S\Sigma$. Then for all $k$ one has
$$
\sigma^{N*}_k(\Sigma,\partial^S\Sigma, [g]) \geq \sigma^{N*}_k(\Omega,\partial^S\Omega, [g|_{\overline\Omega}]).
$$
\end{proposition}

As a corollary of Proposition~\ref{subdomain} we get
\begin{corollary}\label{Neumann cor2}
Let $(M, g)$ be a compact Riemannian surface with boundary. Consider a sequence $\{ K_n \}$ of smooth domains $K_n \subset M$ such that
\begin{itemize}
\item $K_r \subset K_s$ $\forall r>s$;
\item $\cap_n K_n=\{p_1,\ldots,p_l\}$ for some points $p_1,\ldots,p_l \in M$.
\end{itemize} 
Then one has
$$
\lim_{n \to \infty}\sigma^{N*}_k(M \setminus K_n, \partial M \setminus \partial K_n, [g])= \sigma^*_k(M, [g]).
$$ 
\end{corollary}
The proof is postponed to Section \ref{appendix2}.

\subsection{Disconnected surfaces.}
The proofs of two lemmas below follow the exactly same arguments as the proofs of Lemma 4.9 and Lemma 4.10 in \cite{karpukhin2019friedlander}. Their proofs are postponed to Section \ref{appendix2}.
\begin{lemma}
\label{disconnected}
Let $(\Omega,g) = \sqcup_{i=1}^s(\Omega_i,g_i)$ be a disjoint union of Riemannian surfaces with Lipschitz boundary. Set $\partial^S\Omega=\sqcup_{i=1}^s\partial^S\Omega_i$. Then for all $k>0$ one has
$$
\sigma^{N*}_k(\Omega,\partial^S\Omega, [g]) = \max_{\sum\limits_{i=1}^s k_i=k,\,\,\,k_i>0}\,\,\sum_{i=1}^s\sigma^{N*}_{k_i}(\Omega_i,\partial^S\Omega_i, [g_i]).
$$
\end{lemma}

\begin{lemma}\label{omega_i}
Let $(\Sigma,g)$ be a Riemannian surface with boundary. Consider a set of pairwise disjoint Lipschitz domains $\{\Omega_i\}^s_{i=1}$ in $\Sigma$ such that $\Sigma=\bigcup^s_{i=1} \overline\Omega_i$ and $\Omega_i\cap \partial\Sigma=\partial^S\Omega_i \neq \O$ for $1 \leq i \leq s'$. Then one has
$$
\sigma^{*}_k(\Sigma, [g]) \geq \max_{\sum_{i=1}^{s'} k_i=k,\,\,\,k_i \geq 0} \sum^{s'}_{i=1} \sigma^{N*}_{k_i}(\Omega_i,\partial^S\Omega_i, [g]).
$$
\end{lemma}

\section{Proof of Theorem~\ref{non-bound}.} \label{appendix4}
The proof is inspired by the methods of the papers~\cite{MR577325,girouard2012upper, MR3579963}. Let $\Sigma$ be a non-orientable compact surface of genus $\gamma$ and $l$ boundary components. We pass to its orientable cover $\pi\colon\widetilde\Sigma \to \Sigma$. Note that $\Sigma$ is of genus $\gamma$ and has $2l$ boundary components. Let $\tau$ denote the involution exchanging the sheets of $\pi$. If $h$ is a metric on $\Sigma$ then $g:=\pi^*h$ is a metric on $\widetilde\Sigma$ invariant with respect to $\tau$, i.e. $\tau$ is an isometry of $g$. Let $\mathcal D_{\widetilde\Sigma}$ be the Dirichlet-to-Neumann map acting on functions on $\widetilde\Sigma$. Then $\tau\circ\mathcal D_{\widetilde\Sigma}=\mathcal D_{\widetilde\Sigma}\circ\tau$ and hence Steklov eigenfunctions are divided into $\tau-$odd and $\tau-$even ones. The corresponding Steklov eigenvalues are also divided into odd and even ones. Let $\sigma^\tau_k(\widetilde\Sigma,g)$ the $k-$th $\tau-$even Steklov eigenvalue. Then $\sigma^\tau_k(\widetilde\Sigma,g)=\sigma_k(\Sigma,h)$. 

By a well-known theorem of Ahlfors~\cite{ahlfors1950open} there exists a proper conformal branched cover $\psi\colon(\widetilde\Sigma,g) \to (\mathbb D^2,g_{can})$. The word "proper" means $\psi(\partial\widetilde\Sigma)=\mathbb S^1$. Let $d$ be its degree. Define the following pushed-forward metric $g^*$ on $\mathbb D^2$: consider a neighbourhood $U$ of a non-branching point $p\in\mathbb D^2$. Its pre-image is a collection of $d$ neighbourhoods $U_i, i=1,\ldots,d$ on $\widetilde\Sigma$. Moreover, $\psi_i:=\psi_{|_{U_i}}\colon U_i\to U$ is a diffeomorphism. Then the metric $g^*$ is defined on $U$ as $\sum(\psi^{-1}_i)^*g$. The metric $g^*$ is a metric on $\mathbb D^2$ with isolated conical singularities at branching points of $\psi$. The following lemma is trivial

\begin{lemma}\label{Yau}
For any function $u\in C^\infty(\mathbb D^2)$ one has 
$$
\int_{\mathbb S^1}udv_{g^*}=\int_{\partial\widetilde\Sigma}(\psi^*u)dv_{g}
$$
and
$$
d\int_{\mathbb D^2}|\nabla_{g^*} u|^2dv_{g^*}=\int_{\widetilde\Sigma}|\nabla_{g} (\psi^*u)|^2dv_{g}.
$$
\end{lemma}

Further, suppose that there exists an involution $\iota$ of $\mathbb D^2$ such that
\begin{gather}\label{condition}
\psi \circ \tau= \iota \circ \psi.
\end{gather}

\begin{lemma}
The involution $\iota$ is an isometry of $(\mathbb D^2,g^*)$.
\end{lemma}

\begin{proof}
Indeed, let the neighbourhood $U\subset \mathbb D^2$ be small enough and do not contain branching points. Then $\psi^{-1}(U)=\sqcup^d_{i=1} U_i$ and applying $\tau$ one gets: $\tau(\psi^{-1}(U))=\sqcup^d_{i=1} \tau(U_i)$. Note that condition~\eqref{condition} implies $\tau(\psi^{-1}(U))=\psi^{-1}(\iota(U))$. Whence $\psi^{-1}(\iota(U))=\sqcup^d_{i=1} \tau(U_i)$. Let $\widetilde{\psi_i}:=\psi_{\tau(U_i)}$. Then on $U$ one has
\begin{gather*}
g^*=\sum^d_{i=1}(\widetilde{\psi_i}^{-1})^*g=\sum^d_{i=1}(\widetilde{\psi_i}^{-1})^*\tau^*g=\sum^d_{i=1}(\widetilde{\psi_i}^{-1}\circ\tau)^*g=\sum^d_{i=1}(\iota\circ\widetilde{\psi_i}^{-1})^*g=\\=\sum^d_{i=1}\iota^*(\widetilde{\psi_i}^{-1})^*g=\iota^*g^*.
\end{gather*}
\end{proof}

Consider a $j-$th $\iota-$even eigenfunction $u_j$ on $(\mathbb D^2,g^*)$ with corresponding eigenvalue $\sigma^\iota_j(\mathbb D^2,g^*)$. Then the function $\psi^*u_j$ on $\widetilde\Sigma$ is $\tau-$even and hence it projects to a well-defined function $v_j$ on $\Sigma$. We can construct the following function $v=\sum_{j=0}^{k-1}c_jv_j$. Note that $\pi^*v=\sum_{j=0}^{k-1}c_j\psi^*u_j=\psi^*u$, where $u:=\sum_{j=0}^{k-1}c_ju_j$. Further, let $w_i$ denote an $i-$th eigenfunction on $\Sigma$ with eigenvalue $\sigma_i(\Sigma,h)$. It is easy to see that one can always find some coefficients $c_0,\ldots,c_{k-1}$ such that $\int_{\partial\Sigma}v w_idv_h=0, i=0,\ldots,k-1$. Then we can use $v$ as a test function for $\sigma_k(\Sigma,h)$:
$$
\sigma_k(\Sigma,h) \leq \frac{\int_{\Sigma}|\nabla_{h} v|^2dv_{h}}{\int_{\partial\Sigma}v^2dv_{h}}=\frac{\int_{\widetilde\Sigma}|\nabla_{g} \psi^*u|^2dv_{g}}{\int_{\partial\widetilde\Sigma}(\psi^*u)^2dv_{g}}=d\frac{\int_{\mathbb D^2}|\nabla_{g^*} u|^2dv_{g^*}}{\int_{\mathbb S^1}u^2dv_{g^*}}=d\sigma^\iota_k(\mathbb D^2,g^*),
$$
where we used Lemma~\ref{Yau}. Moreover, the second identity in Lemma~\ref{Yau} implies $L_{g^*}(\mathbb S^1)=L_g(\partial\widetilde\Sigma)=2L_h(\partial \Sigma)$. Whence
\begin{gather}\label{need}
\overline\sigma_k(\Sigma,h) \leq \frac{d}{2}\sigma^\iota_k(\mathbb D^2,g^*)L_{g^*}(\mathbb S^1).
\end{gather}

Consider a conformal map $\psi$ between surfaces with involution $\psi\colon (\widetilde\Sigma, \tau) \to (\mathbb D^2, \iota)$ of minimal degree $d$. The map $\psi$ is conformal, moreover every involution exchanging the orientation on $\mathbb D^2$ is conjugate to the involution $\iota_0(z):=\bar z$, where we identify $\mathbb D^2$ with the unit disc on the complex plane. Therefore, without loss of generality we can assume that $\iota=\iota_0$. The fixed point set of $\iota_0$ is the diameter $\{z\in \mathbb D^2~|~Re(z)=0\}$. Let $H\mathbb D^2$ denote a half-disc for example the right one and $\partial^SH\mathbb D^2$ is the right half-circle. Thus, $\sigma^{\iota_0}_k(\mathbb D^2,g^*)=\sigma^N_k(H\mathbb {D}^2,\partial^SH\mathbb{D}^2, g^*)$ and inequality~\eqref{need} implies:
\begin{equation}\label{done}
\begin{split}
\overline\sigma_k(\Sigma,h) \leq \frac{d}{2}\sigma^\iota_k(\mathbb D^2,g^*)L_{g^*}(\mathbb S^1)=d\overline\sigma^N_k(H\mathbb {D}^2,\partial^SH\mathbb{D}^2, g^*) \leq \\ \leq d\sigma^{N*}_k(H\mathbb {D}^2,\partial^SH\mathbb{D}^2, [g^*]) \leq d\sigma^{*}_k(\mathbb {D}^2, [g_{can}])=2\pi k d,
\end{split}
\end{equation}
where in the last inequality we used Lemma~\ref{subdomain} and the fact that there exists a unique up to an isometry conformal class $[g_{can}]$ on $\mathbb D^2$. We want to estimate $d$ in formula~\eqref{done}. It is known that there exists a proper conformal branched cover $f\colon(\widetilde\Sigma, g) \to (\mathbb D^2,g_{can})$ of degree $d'\leq\gamma+2l$ (see \cite{gabard2006representation}). One can construct the following map $F(x):=f(x)\bar f(\tau(x))$. Note that $\bar F(x)=F(\tau(x))=\iota (F(x))$ and hence $\iota=\iota_0$. Moreover $F$ is proper and the degree of $F$ is not greater than $2d'=2(\gamma+2l)$. Hence there exists a proper map between $(\widetilde\Sigma, \tau)$ and $(\mathbb D^2,\iota_0)$ of degree not exceeding $2d'=2(\gamma+2l)$ satisfying \eqref{condition}. Inequality~\eqref{done} then implies
$$
\overline\sigma_k(\Sigma,h) \leq 4\pi k (\gamma+2l).
$$

\section{Geometric background}
\label{geometry}
The aim of this section is the proof of Theorem~\ref{conf&conv}. For this purpose we provide a necessary background concerning the geometry of moduli space of conformal classes on a surface with boundary. We start with closed orientable surfaces.

\subsection{Closed orientable surfaces} 

Let us recall the \textit{Uniformization theorem}.

\begin{theorem}
Let $\Sigma$ be a closed surface and $g$ be a Riemannian metric on it. Then in the conformal class $[g]$ there exists a unique (up to an isometry) metric $h$ of constant Gauss curvature and fixed area. The area assumption is unnecessary except in the case of the torus for which we fix the volume of $h$ to be equal to $1$
\end{theorem}

\begin{remark} 
It follows from the Gauss-Bonnet theorem that the metric $h$ in the Uniformization theorem is of Gauss curvature $1$ in the case of the sphere, $0$ in the case of the torus and $-1$ in the rest cases. 
\end{remark}


Recall that a Riemannian metric $h$ of constant Gaussian curvature $-1$ is called {\em hyperbolic} and a Riemannian surface $(\Sigma,h)$ endowed with a hyperbolic metric $h$ is called {\em a hyperbolic surface}. Note also that a hyperbolic surface is necessarily of negative Euler characteristic. We also say that the torus endowed with a metric of curvature $h=0$ is a flat torus and the sphere endowed with the metric $h=1$ is the standard (round) sphere.


\subsection{Hyperbolic surfaces} 

 We recall that a \textit{pair of pants} is a compact surface of genus $0$ with $3$ boundary components. The following theorem plays an underlying role in the theory of hyperbolic surfaces. 

\begin{theorem}[Collar theorem (see e.g.~\cite{MR1183224})]\label{Collar theorem}
Let $(\Sigma,h)$ be an orientable compact hyperbolic surface of genus $\gamma \geq 2$ and let $c_1,c_2,\ldots,c_m$ be pairwise disjoint simple closed geodesics on $(\Sigma,h)$. Then the following holds
\begin{enumerate}[(i)]
\item $m \leq 3 \gamma-3$.
\item There exist simple closed geodesics $c_{m+1},\ldots,c_{3 \gamma-3}$ which, together with $c_1,\ldots,c_m$, decompose $\Sigma$ into pairs of pants.
\item The collars 
\begin{align*}
\mathcal{C}(c_i)=\left\{p\in\Sigma~|~ dist(p,c_i) \leq w(c_i)\right\}
\end{align*}
of widths 
\begin{align*}
w(c_i)=\frac{\pi}{l(c_i)}\left(\pi-2\arctan\left(\sinh\frac{l(c_i)}{2}\right)\right)
\end{align*}
are pairwise disjoint for $i=1,\ldots,3 \gamma-3$.
\item
Each $\mathcal{C}(c_i)$ is isometric to the cylinder 

$$\left\{(t,\theta)| -w(c_i)<t<w(c_i),\,\theta\in\mathbb{R}/2\pi\mathbb{Z}\right\}$$
with the Riemannian metric 

\begin{align*}
\left(\frac{l(c_i)}{2\pi \cos\left(\frac{l(c_i)}{2\pi}t\right)}\right)^2\left(dt^2+d\theta^2\right).
\end{align*}

\end{enumerate}
\end{theorem}


The decomposition of $(\Sigma,h)$ into pair of pants which we denote by $\mathcal{P}$ is called \textit{the pants decomposition}. We also say that the geodesics $c_1,\ldots,c_{3 \gamma-3}$ form $\mathcal{P}$. 

\subsection{Convergence of hyperbolic metrics} 
We endow the set of hyperbolic metrics on a given surface $\Sigma$ with $C^\infty-$topology. In this section we describe the convergence on this topological set which is called \textit{the moduli space of conformal classes} on $\Sigma$. Essentially, two cases can happen: the injectivity radii of a sequence of hyperbolic metrics do not go to $0$ or they do. The first case is described by \textit{Mumford's compactness theorem} and the second one is treated by \textit{the Deligne-Mumford compactification}.

\begin{proposition}[Mumford's compactness theorem (see e.g. ~\cite{MR1451624})]
\label{Mumford}
Let $\{h_n\}$ be a sequence of hyperbolic metrics on a surface $\Sigma$ of genus $\geq 2$. Assume that the injectivity radii $\inj(\Sigma,h_n)$ satisfy $\limsup\limits_{n\to\infty}\inj(\Sigma,h_n)>0$. Then there exists a subsequence $\{h_{n_k}\}$, sequence $\{\Phi_k\}$ of smooth automorphisms of $\Sigma$ and a hyperbolic metric $h_\infty$ on $\Sigma$ such that the sequence of hyperbolic metrics $\{\Phi_k^*h_{n_k}\}$ converges in $C^\infty$-topology to $h_\infty$.
\end{proposition}

If $\lim\limits_{n\to\infty}\inj(\Sigma,h_n)=0$ then we say that the sequence $\{h_n\}$ {\em degenerates}. The thick-thin decomposition implies that if the sequence $\{h_n\}$ degenerates then for each $n$ there exists a collection $\{c_1^n,\ldots,c_s^n\}$ of disjoint simple closed geodesics in $(\Sigma,h_n)$ whose lengths tend to $0$ and the length of any geodesic in the complement $\Sigma_n =\Sigma\backslash (c_1^n\cup\ldots\cup c_s^n)$ is bounded from below by a constant independent of $n$. We call the geodesics $\{c_1^n,\ldots,c_s^n\}$ "pinching" or "collapsing". The surface $(\Sigma_n, h_n)$ is possibly a disconnected hyperbolic surface with geodesic boundary. Let $\widehat{\Sigma_\infty}$ denote the surface having the same connected components as $\Sigma_n$, but with boundary component replaced by marked points. Note that each sequence $\{c_i^n\}$ corresponds to a pair of marked points $\{p_i,q_i\}$ on $\widehat{\Sigma_\infty}$, $i=1,\ldots,s$. Then the punctured surface $\widehat{\Sigma_\infty}\backslash\{p_1,q_1,\ldots,p_s,q_s\}$ that we denote by $\Sigma_\infty$ admits the unique hyperbolic  metric $h_\infty$ with cusps at punctures. Now we are ready to formulate one of the underlying results in the theory of \textit{moduli spaces of Riemann surfaces}.

\begin{proposition}[Deligne-Mumford compactification (see e.g. ~\cite{MR1451624})]\label{D-M} 
Let $(\Sigma, h_n)$ be a sequence of hyperbolic surfaces such that $\inj(\Sigma,h_n)\to 0$. Then up to a choice of subsequence, there exists a sequence of diffeomorphisms $\Psi_n: \Sigma_\infty \to \Sigma_n$ such that the sequence $\{\Psi^*_n h_n\}$ of hyperbolic metrics  converges in $C_{\mathrm{loc}}^\infty$-topology to the complete hyperbolic metric $h_\infty$ on $\Sigma_\infty$. Furthermore, there exists a metric of locally constant curvature $\widehat{h_\infty}$ on $\widehat{\Sigma_\infty}$ such that its restriction to $\Sigma_\infty$ is conformal to $h_\infty$.
\end{proposition}

We call $(\widehat{\Sigma_\infty},\widehat{h_\infty})$ a {\em limiting space} of the sequence $(\Sigma,h_n)$. We also say that the limit of conformal classes $[h_n]$ is the conformal class $[\widehat{h_\infty}]$ on $\widehat{\Sigma_\infty}$.

\begin{remark}
We emphazise that $\widehat{h_\infty}$ has {\em locally} constant curvature, since $\widehat{\Sigma_\infty}$ is possibly disconnected and different connected components could have different signs of Euler characteristic. 
\end{remark}

 

\subsection{Orientable surfaces with boundary of negative Euler characteristic}
\label{modulii_boundary}
Our exposition of this topic essentially follows the book~\cite{jost2007bosonic}. 

Let $\Sigma$ be an orientable surface of genus $\gamma$ with $l$ boundary components. Consider its \textit{Schottky double} $\Sigma^d$ defined in following way. We identify $\Sigma$ with another copy $\Sigma'$ of $\Sigma$ with opposite orientation along the common boundary. We get a closed oriented surface of genus $2\gamma+l-1$. For example the Schottky double of the disk is the sphere and the Schottky double of the cylinder is the torus. In the rest cases we always get a hyperbolic surface as the Schottky double. We endow the surface $\Sigma$ with a metric $g$. The next theorem plays a role of the Uniformization theorem for surfaces with boundary.

\begin{proposition}[\cite{osgood1988extremals}] \label{uniformization}
In the conformal class $[g]$ of a metric $g$ on the surface $\Sigma$ there exists a unique (up to an isometry) metric of constant Gauss curvature and geodesic boundary. More precisely, this metric is of curvature $1$ in the case of $\mathbb D^2$, of the curvature $0$ in the case of the cylinder and of curvature $-1$ in the rest cases.
\end{proposition}  

Denote the metric of constant Gauss curvature and geodesic boundary from Theorem \ref{uniformization} by $h$. Consider a Riemannian surface with boundary $(\Sigma, h)$. Its Schottky double admits the metric $h^d$ defined as $h^d_{|_\Sigma}=h$ and $h^d_{|_\Sigma'}=h$. It is a metric of constant curvature and the involution $\iota: \Sigma^d \to \Sigma^d$ that interchanges $\Sigma$ and $\Sigma'$ becomes an isometry with $\partial \Sigma$ as the fixed set. Moreover, $(\Sigma,h_n)=(\Sigma^d,h^d_n)/\iota$. 


Theorem \ref{uniformization} also says that the set of conformal classes on the surface $\Sigma$ with boundary is in one-to-one correspondence with the set of metrics of constant Gauss curvature and geodesic boundary which is in the one-to-one correspondence with the set of "symmetric" metrics (metrics that go to themselves under the involution $\iota$) of constant curvature on the Schottky double. We endow the set of metrics of constant Gauss curvature and geodesic boundary with $C^\infty-$topology. Consider a sequence of conformal classes $\{c_n\}$ on $\Sigma$. It uniquely defines a sequence of "symmetric" metrics of constant curvature $\{h^d_n\}$ on $\Sigma^d$. For this sequence we have the same dichotomy as we have seen in the previous sections. Precisely, either $\inj (\Sigma^d,h^d_n) \nrightarrow0$ or $\inj (\Sigma^d,h^d_n)\to 0$. In the first case we get a genuine Riemannian metric on $\Sigma^d$ which is obviously "symmetric" and of constant curvature while in the second case one can find a set of simple closed geodesics $\{c_1^n,\dots,c_s^n\}$ where $s \leq 6\gamma+3l-6$ whose lengths $l_{h^d_n}(c_i^n)\to 0$. For the geodesics $c_i^n$ there exist two possibilities: either $\iota(c_i^n)=c_i^n$ or $\iota(c_i^n)=c_j^n$ with $j \neq i$. The first possibility implies that the geodesic $c_i^n$ crosses $\partial \Sigma$ which corresponds to two situations as well: either $c_i^n$ has exactly two points of intersection with $\partial \Sigma$ or it belongs to $\partial \Sigma$, i.e. it is one of the boundary components. The second possibility implies that $c_i^n$ does not crosse $\partial \Sigma$. Taking quotient by $\iota$ we then get three types of pinching geodesics on $(\Sigma,h_n)$ with $\inj (\Sigma,h_n) \to 0$: pinching boundary components, pinching simple geodesics which have exactly two points of intersection with the boundary and pinching simple closed geodesics which do not cross the boundary. 

\subsection{Non-orientable surface with boundary of negative Euler characteristic}\label{nonor} Let $\Sigma$ be a compact non-orientable surface with $l$ boundary components. Note that the Uniformization Theorem \ref{uniformization} also holds for non-orientable surfaces. Pick a metric $h$ of constant Gauss curvature and geodesic boundary. We pass to the orientable cover that we denote by $\widetilde\Sigma$. The surface $\widetilde\Sigma$ is a compact orientable surface with $2l$ boundary components. The pull-back of the metric $h$ that we denote by $\tilde h$ is a metric of constant Gauss curvature and with geodesic boundary. Moreover, this metric is invariant under the involution changing the orientation on $\widetilde\Sigma$. Consider a sequence $\{h_n\}$ on $\Sigma$ of metrics of constant Gauss curvature and geodesic boundary such that $\inj(\Sigma,h_n)\to 0$ as $n\to\infty$. This sequence corresponds to the sequence $\{\tilde h_n\}$ on $\widetilde\Sigma$ such that $\inj(\widetilde\Sigma,\tilde h_n)\to 0$ as $n\to\infty$. As we discussed in the previous section for the sequence $\{\tilde h_n\}$ one can find pinching geodesics of the following three types: pinching boundary components, pinching simple geodesics crossing the boundary at two points and pinching simple closed geodesics which do not cross the boundary. Note that for the geodesics of the second type the points of intersection with the boundary are not identified under the involution. Indeed, if the were identified then the corresponding pinching geodesic had fixed ends under the involution. Applying the involution to this geodesic we would get a pinching \textit{closed} geodesic crossing the boundary at two points which is not one of the possible types of pinching geodesics. Consider now the geodesics of the third type. For every such geodesic there are two possible cases: either this geodesic maps to itself under the involution changing the orientation or it maps to another simple closed geodesic which does not cross the boundary. Then taking the quotient by the involution changing the orientation we get two types of simple closed geodesics on $\Sigma$ which do not crosse the boundary: \textit{one-sided geodesics} which are the images of the geodesics described in the first case and \textit{two-sided geodesics} which are the images of the geodesics described in the second case. The collars of one-sided geodesics are nothing but M\"obius bands while the collars of two-sided geodesics are cylinders. Therefore, if $\inj(\Sigma,h_n)\to 0$ as $n\to\infty$ then one can find pinching geodesics of the following types: pinching boundary components, pinching simple geodesics which have exactly two points of intersection with the boundary, one-sided pinching simple closed geodesics not crossing the boundary and two-sided pinching simple closed geodesics not crossing the boundary.

\subsection{Surfaces with boundary of non-negative Euler characteristic} Here we consider the cases of the disc, the cylinder $\mathcal C$ and the M\"obius band $\mathbb{MB}$. 

It is known that the disc has a unique conformal class (up to an isometry). We denote this conformal class as $[g_{can}]$ or $c_{can}$, where $g_{can}$ is the flat metric on the disc $\mathbb D^2$ with unit boundary length. 

Accordingly to Theorem~\ref{uniformization} in a conformal class on $\mathcal C$ there exists a flat metric with geodesic boundary, i.e. a metric on the right circular cylinder. This metric is unique if we fix the length of the boundary. The right circular cylinder is uniquely determined by its height. Therefore, conformal classes on $\mathcal C$ are in one-to-one correspondence with heights of right circular cylinders, i.e. the set of conformal classes is $\mathbb R_{>0}$. We will identify conformal classes on $\mathcal C$ with points of $\mathbb R_{>0}$. We say that the sequence $\{c_n\}$ of conformal classes degenerates if either $c_n \to 0$ or $c_n\to\infty$. The case $c_n \to 0$ corresponds to a pinching geodesic having intersection with two boundary components (i.e. the generatrix of the right circular cylinder). The case $c_n\to\infty$ corresponds to pinching boundary components. 

In the case of the M\"obius band we also use Theorem~\ref{uniformization} which implies that in every conformal class on $\mathbb{MB}$ there exists a flat metric with geodesic boundary which is unique if we fix the length of the boundary. Passing to the orientable cover and pulling back the flat metric from $\mathbb{MB}$ we get a flat cylinder with geodesic boundary. Then the discussion in the previous paragraph implies that the conformal classes on $\mathbb{MB}$ are also encoded by $\mathbb R_{>0}$. Identifying again conformal classes on $\mathbb{MB}$ with points of $\mathbb R_{>0}$ we get two possible cases for a sequence of conformal classes $\{c_n\}$: either $c_n \to 0$ or $c_n\to\infty$. In both cases we say that the sequence $\{c_n\}$ degenerates. The first case corresponds to a pinching geodesic having two points of intersection with boundary. The second case corresponds to the collapsing boundary. 

\medskip

\section{Proof of Theorem \ref{conf&conv}.} \label{proofconf&conv}
{\bf Negative Euler characteristic.} Let $\Sigma$ be a surface with boundary and $c_n\to c_\infty$ a degenerating sequence of conformal classes. Consider the corresponding sequence of metrics $h_n$ of constant Gauss curvature and geodesic boundary. Then as we have noticed in Subsection \ref{modulii_boundary} one can find $s=s_1+s_2+s_3$ pinching geodesics of the following three types: $s_1$ pinching boundary components, $s_2$ pinching geodesics that have two points of intersection with boundary and $s_3$ pinching simple closed geodesics that do not intersect the boundary. 

We introduce the following notations
\begin{itemize}
\item $\gamma^n_i$ for collapsing geodesics, $i=1,\ldots,s$. If we do not indicate the superscript then the symbol $\gamma_i$ stands for the genus;
\item $\mathcal{C}^n_i$ for collars of collapsing geodesics, $i=1,\ldots,s$. Their widths are denoted by $w^n_i$. Moreover, $\mathcal{C}^n_i:=\{(t,\theta)~|~0 \leq t<w^n_i,~0 \leq\theta \leq 2\pi\}$ for $1 \leq i \leq s_1$ and $\mathcal{C}^n_i:=\{(t,\theta)~|~-w^n_i<t<w^n_i,~0 \leq\theta \leq 2\pi\}$ for $s_1+1 \leq i \leq s$ (if the geodesic is one-sided then we consider $\mathcal{C}^n_i:=\{(t,\theta)~|~-w^n_i< t<w^n_i,~0 \leq\theta \leq 2\pi\}/\sim$, where $\sim$ stands for $(t,\theta)\sim(-t,\pi+\theta)$). Note that geodesics correspond to the line $\{t=0\}$, the segments $\{\theta=0\}$ and $\{\theta=2\pi\}$ are identified for $1 \leq i \leq s_1$ and for $s_1+s_2+1 \leq i \leq s$ and they are not identified for $s_1+1 \leq i \leq s_1+s_2$ and correspond to the segments of intersection with the boundary; 
\item for $0<a<w^n_i$, we denote $\mathcal C^n_i(0,a)$ the subset $\{(t, \theta)\,|~0 \leq t \leq a, 0 \leq \theta \leq 2\pi\}\subset {\mathcal{C}}^n_i$  for $1 \leq i \leq s_1$ and for $-w^n_i< a< b< w^n_i$, we denote $\mathcal C^n_i(a,b)$ the subset $\{(t, \theta)\,|~a \leq t \leq b, 0 \leq \theta \leq 2\pi\}\subset {\mathcal{C}}^n_i$ for $s_1+1 \leq i \leq s$;
\item $\Gamma^n_i:=\{(t,\theta)\in\mathcal{C}^n_i~|~\theta=0$ or $\theta=2\pi\}$ for $s_1+1 \leq i \leq s_1+s_2$;
\item for $-w^n_i< a< b< w^n_i$, we set $\Gamma^n_i(a,b):=\{(t,\theta)\in\Gamma^n_i~|~a \leq t \leq b\}$ for $s_1+1 \leq i \leq s_1+s_2$;
\item $\Sigma^n_j$ for the $j-$th connected component of $\Sigma \setminus \cup^s_{i=1} \mathcal{C}^n_i$. We enumerate $\Sigma^n_j$ by $1 \leq j \leq M$ such that $M$ denotes the number of $\Sigma^n_j$ and for all $1 \leq j \leq m$ one has $\Sigma^n_j \cap \partial \Sigma \neq \O$ ;
\item let $\alpha^n = \cup_{i=1}^{s_1+s_2}\{\alpha^n_{i,-},\alpha^n_{i,+}\}$, where $0 \leq\alpha^n_{i,\pm}< w^n_i$. We denote by $\Sigma^n_j(\alpha^n)$ the connected component of
 $$
 \Sigma \setminus \Big(\bigcup^{s_1+s_2}_{i=1}{\mathcal{C}}_i^n(\alpha^n_{i,-},\alpha^n_{i,+}) \cup \bigcup^{s}_{i=s_1+s_2+1} \gamma^n_i\Big)
 $$
which contains $\Sigma_j^n$;
\item for $\alpha^n = \cup_{i=1}^{s_1+s_2}\{\alpha^n_{i,-},\alpha^n_{i,+}\}$, where $0 \leq\alpha^n_{i,\pm}< w^n_i$ we set $I^n_j(\alpha^n)=\Sigma^n_j(\alpha^n)\cap \partial\Sigma$ and $I^n_j=\Sigma^n_j \cap \partial\Sigma$ where $1 \leq j \leq m$;
\item we use the notation $a_n \ll b_n$ for two sequences $\{a_n\}$ and $\{b_n\}$ satisfying $a_n, b_n \to +\infty$  and $\frac{a_n}{b_n} \to 0$ as $n \to \infty$. 
\end{itemize}

\subsection{Inequality $\geq$.} We prove that

\begin{equation}\label{aim}
\begin{split}
\liminf_{n \to \infty} \sigma^*_k (\Sigma, c_n) \geq \max \Big(\sum^{m}_{i=1} \sigma^*_{k_{i}}(\Sigma_{\gamma_{i},l_i},c_\infty)+ \sum_{i=1}^{s_1+s_2}\sigma^*_{r_i}(\mathbb{D}^2) \Big),
\end{split}
\end{equation}

For this aim we consider the domains ${\mathcal{C}}^n_i(0,\alpha_{i,+}^n)$ for $1 \leq i \leq s_1$, ${\mathcal{C}}^n_i(\alpha_{i,-}^n,\alpha_{i,+}^n)$ for $1+s_1 \leq i \leq s_1+s_2$, where  $ w^n_i - \alpha_{i,\pm}^n \ll w^n_i,$ $\alpha_{i,\pm}^n\to\infty$ and the domains $\Sigma_j^n(\alpha^n)$ for $1 \leq j \leq m$. By Lemma~\ref{omega_i} we have
\begin{equation}
\label{aim2}
\begin{split}
&\sigma^*_k(\Sigma, c_n) \geq \max \Big(\sum^{s_1}_{i=1}\sigma^{N*}_{r_{i}}({\mathcal{C}}^n_i(0,\alpha_{i,+}^n), \gamma^n_i, c_n)+\\&+\sum^{s_1+s_2}_{i=1+s_1}\sigma^{N*}_{r_{i}}({\mathcal{C}}^n_i(\alpha_{i,-}^n,\alpha_{i,+}^n),\Gamma^n_i(\alpha_{i,-}^n,\alpha_{i,+}^n), c_n)+\sum^{m}_{j=1}\sigma^{N*}_{k_{j}}(\Sigma_j^n(\alpha^n), I^n_j(\alpha^n), c_n)\Big).
\end{split}
\end{equation}

 For $1 \leq i \leq s_1$ we define the conformal maps $\Psi_i^n\colon ({\mathcal{C}}^n_i(0,\alpha_{i,+}^n), c_n) \to (\mathbb{D}^2, [g_{can}])$ as
$$
\Psi_i^n(t,\theta)= e^{\sqrt{-1}(\theta+\sqrt{-1}t)}.
$$ 
The images of $\Psi_i^n$ are the annuli  $\mathbb D^2 \setminus \mathbb D^2_{e^{-\alpha_{i,+}^n}}$ exhausting $\mathbb D^2$ as $n \to \infty.$ We also note that $\Psi_i^n(\gamma^n_i)=\mathbb S^1$.

For $s_1+1 \leq i \leq s_1+s_2$ we define the conformal maps $\Psi_i^n\colon ({\mathcal{C}}^n_i(\alpha_{i,-}^n,\alpha_{i,+}^n), c_n) \to (\mathbb{D}^2, [g_{can}])$ as
$$
\Psi_i^n(t,\theta)=\tan\Big(\frac{\theta-\pi+\sqrt{-1}t}{4}\Big).
$$
The images of $\Psi_i^n$ that we denote by $\Omega_{i}^n$ exhaust $\mathbb D^2$ as $n \to \infty.$  We also denote the image of $\Gamma^n_i(\alpha_{i,-}^n,\alpha_{i,+}^n)$ by $\partial^S\Omega_{i}^n$. Note that $\partial^S\Omega_{i}^n$ exhaust $\mathbb S^1$ as $n\to\infty$.
  
Finally, we take restrictions of the diffeomorphisms $\Psi_n^{-1}$ given by Proposition~\ref{D-M} to obtain the conformal maps $\check\Psi_j^n\colon({\Sigma}_j^n(\alpha^n),c_n)\to (\Sigma_{\infty},\Psi_n^*c_n)$ where $1 \leq j \leq m$. Let $\check\Omega_{j}^n \subset \Sigma_\infty$ be the the image of $\check\Psi_j^n$ and $\partial^S\check\Omega_{j}^n:=\check\Psi_j^n(I^n_j(\alpha^n))$. The following lemma holds


\begin{lemma}\label{post1}
Let $\Sigma^\infty_j$ be the connected component $\check\Psi^n_j(\Sigma^n_j)\subset \Sigma_\infty$ where $1 \leq j \leq m$. Then the domains $\check\Omega_j^n$ exhaust $\Sigma^\infty_j$ and $\partial^S\check\Omega_{j}^n$ exhaust $\partial\Sigma^\infty_j$. 
\end{lemma}

\begin{proof}
Passing to the Schottky double of the surface $\Sigma$ we immediately deduce this lemma from \cite[Lemma 5.1]{karpukhin2019friedlander}.
\end{proof}

Further, we apply the conformal transformations to~\eqref{aim2} to get
\begin{equation}\label{aim3}
\begin{split}
&\sigma^*_k(\Sigma, c_n) \geq  \max\Big(\sum^{s_1}_{i=1}\sigma^{N*}_{r_{i}}(\mathbb D^2 \setminus\mathbb D^2_{e^{-\alpha_{i,+}^n}}, \mathbb S^1, [g_{can}])+\\&+\sum^{s_1+s_2}_{i=1+s_1}\sigma^{N*}_{r_{i}}(\Omega_{i}^n, \partial^S\Omega_{i}^n, [g_{can}])+\sum^{m}_{j=1}\sigma^{N*}_{k_{j}}(\check\Omega_j^n, \partial^S\check\Omega_{j}^n, [(\Psi^n)^*h_n])\Big).
\end{split}
\end{equation}
It follows from Corollary~\ref{Neumann cor2} that the first two terms on the right hand side converge to $\sigma_{r_{i}}(\mathbb D^2, [g_{can}])$. To complete the proof we will need the following lemma
\begin{lemma}\label{post2}
Let $\widehat{\Sigma_j^\infty}\subset \widehat{\Sigma_\infty}$ be a closure of $\Sigma_j^\infty$, $1 \leq j \leq m$. Then for all $r$ one has 
$$
\liminf_{n\to\infty}\sigma^{N*}_{r}(\check\Omega_j^n,\partial^S\check\Omega_{j}^n, [(\Psi^n)^*h_n]) \geq \sigma^*_r(\widehat{\Sigma_j^\infty}, [\widehat{h_\infty}]).
$$
\end{lemma}
We postpone the proof to Section \ref{appendix3}.

Finally, taking $\liminf_{n\to\infty}$ in~\eqref{aim3} completes the proof of~\eqref{aim}.

\subsection{Inequality $\leq$.} We prove the inverse inequality,
\begin{equation}\label{aim'}
\begin{split}
\limsup_{n \to \infty} \sigma^*_k (\Sigma, c_n) \leq \max \Big(\sum^{m}_{i=1} \sigma^*_{k_{i}}(\Sigma_{\gamma_{i},l_i},c_\infty)+ \sum_{i=1}^{s_1+s_2}\sigma^*_{r_i}(\mathbb{D}^2) \Big).
\end{split}
\end{equation}

For this aim we  choose a subsequence $c_{n_m}$ such that $$\lim_ {n_m \to \infty} \sigma^*_k (\Sigma, c_{n_m})=\limsup_{n \to \infty} \sigma^*_k (\Sigma, c_n).$$ Then we relabel the subsequence and denote it by $\{c_n\}$. Therefore, one can choose subsequences without changing the value of $\limsup$.

{\bf Case 1.} Suppose that up to a choice of a subsequence the following inequality holds
\begin{align*}
\sigma^*_k(\Sigma, c_n) > \sigma^*_{k-1}(\Sigma, c_n) +2\pi.
\end{align*}
Then by \cite[Theorem 2]{petrides2019maximizing} in the conformal class $c_n$ there exists a metric $g_n$ of unit boundary length induced from a harmonic immersion with free boundary $\Phi_n$ to some $N(n)$-dimensional ball $\mathbb{B}^{N(n)}$, i.e.
$$
g_n=\frac{\langle \Phi_n, \partial_{\nu_n}\Phi_n\rangle_{h_n}}{\sigma^*_k(\Sigma, c_n)}h_n
$$
and such that $\sigma_k(g_n)=\sigma^*_k(\Sigma, c_n)$. Here the metric $h_n$ is the canonical representative in the conformal class $c_n$. It is known that for any compact surface the multiplicity of $\sigma_k(g_n)$ is bounded from above by a constant depending only on $k$ and the topology of $\Sigma$ (see for instance~\cite{fraser2012minimal, karpukhin2014multiplicity}). Therefore, one can choose the number $N(n)$ large enough such that $N(n)$ does not depend on $n$. 

Assume that for the sequence $\{c_n\}$ the following inequality holds

\begin{equation}\label{gap}
\begin{split}
\limsup_{n \to \infty} \sigma^*_k (\Sigma, c_n) > \max \Big(\sum^{m}_{i=1} \sigma^*_{k_{i}}(\Sigma_{\gamma_{i},l_i},c_\infty)+ \sum_{i=1}^{s_1+s_2}\sigma^*_{r_i}(\mathbb{D}^2) \Big).
\end{split}
\end{equation}

For $1 \leq i \leq s_1$ we consider the conformal map $\Psi^n_i: (\mathcal C^n_i, c_n) \to (\mathbb D^2,[g_{can}])$ defined as $\Psi^n_i(\theta,t)=e^{\sqrt{-1}(\theta+\sqrt{-1}t)}$. The image of this map is nothing but $\mathbb D^2\setminus \mathbb D^2_{e^{-w^n_i}}$ which exhausts $\mathbb D^2$ as $n\to\infty$. The image of a pinching geodesic is $\mathbb S^1$. Then the map $\Phi^n_i:=\Phi_n\circ(\Psi^n_i)^{-1}: \mathbb D^2\setminus \mathbb D^2_{e^{-w^n_i}}\to \mathbb B^N$ satisfies the \textit{bubble convergence theorem for harmonic maps with free boundary} \cite[Theorem 1]{laurain2017regularity}. Hence, there exist a regular harmonic map with free boundary $\Phi_i: \mathbb D^2 \to \mathbb B^N$ and some harmonic extensions of non-constant $1/2-$harmonic maps $\omega^i_1,\dots,\omega^i_{t_i}: \mathbb D^2 \to \mathbb B^N$ such that 
$$
\int_{\mathbb D^2}|\nabla \Phi_i|^2dv_{g_{can}}+\sum^{t_j}_{j=1}\int_{\mathbb D^2}|\nabla \omega^j_{t_i}|^2dv_{g_{can}}=\lim_{n\to\infty}\int_{\gamma^n_i}ds_{g_n}.
$$
We denote $\lim_{n\to\infty}\int_{\gamma^n_i}ds_{g_n}$ by $m_i$.

\begin{proposition} \label{sequences}
 For $s_1+1 \leq i \leq s_1+s_2$ there exist integers $t_{i} \geq 0$, non-negative sequences $\{a_{i,l}^n\}, \{b_{i,l}^n\}$ with $1 \leq l \leq t_{i}$ and a sequence $\{\alpha^n_i\}$ such that 
\begin{gather*}
-w_i^n \ll \alpha_{i,-}^n=b_{i,0}^n \ll a_{i,1}^n \ll b_{i,1}^n \ll \ldots \ll a_{i,t_i}^n \ll b_{i,t_{i}+1}^n \ll a_{i,t_{i+1}}^n=\alpha_{i,+}^n \ll w_i^n
\end{gather*}
and 
$$
m_{i,l}=\lim_{n \to \infty} L_{g_n}(\Gamma_i^n(a_{i,l}^n,b_{i,l}^n))>0.
$$  
Moreover, there exists a set $J \subset \{1,\ldots,m\}$ such that for every $j \in J$ one has
$$
m_{j}=\lim_{n \to \infty} L_{g_n}(I_j^n(\alpha^n))>0
$$
satisfying 
$$
\sum^{s_1}_{i=1} m_{i}+\sum^{s_1+s_2}_{i=1} \sum^{t_i}_{l=s_1+1}m_{i,l}+\sum_{j\in J}m_j=1,
$$
with $s_1+\sum^{s_1+s_2}_{i=s_1+1}t_i$ is maximal.
\end{proposition}
\begin{proof}
The proof follows the proofs of Claim 16, Claim 17 by \cite{petrides2019maximizing}. Precisely, denying the proposition one can construct $k+1$ test-functions such that $\sigma_k(g_n) \leq o(1)$ which contradicts inequality~\eqref{el soufi}. The construction of these functions is given in the proofs of Claim 16, Claim 17 by \cite{petrides2019maximizing}. Note that these functions equal $1$ on $\Sigma^n_j$ for every $m+1 \leq j \leq M$.
\end{proof}

We proceed with considering a sequence $\{d_{i,l}^n\}$ where $s_1+1 \leq i \leq s_1+s_2$ and $1 \leq l \leq t_i$ such that 
$$
\lim_{n \to \infty} L_{g_n}(\Gamma_i^n(a_{i,l}^n,d_{i,l}^n))= \lim_{n \to \infty} L_{g_n}(\Gamma_i^n(d_{i,l}^n,b_{i,l}^n))=m_{i,l}/2.
$$
Let $q^n_{i,l}\ll a^n_{i,l}$, $q^n_{i,l}\to+\infty$. Consider the conformal maps 
$$\Psi_{i,l}^n\colon \left({\mathcal{C}}_i^n(a_{i,l}^n-q^n_{i,l},b_{i,l}^n+ q^n_{i,l}),c_n\right) \to (\mathbb{D}^2,[g_{can}])
$$ defined as
$$
\Psi_{i,l}^n(t,\theta)=\tan\left(\frac{\theta-\pi+\sqrt{-1}(t-t_{i,l}^n)}{4}\right)
$$
Let
$$
D_{i,j}^n=\Psi_{i,l}^n\left({\mathcal{C}}_i^n(a_{i,l}^n-q^n_{i,l},b_{i,l}^n+ q^n_{i,l})\right)
$$
and 
$$
S_{i,j}^n=\Psi_{i,l}^n\left(\Gamma_i^n(a_{i,l}^n-q^n_{i,l},b_{i,l}^n+ q^n_{i,l})\right)
$$
Then $D_{i,j}^n$ exhausts $\mathbb D^2$ and $S_{i,j}^n$ exhausts $\mathbb S^1$ as $n\to \infty$. We also set
$$
\lim_{n\to\infty} L_{(\Psi_{i,l}^n)_*g_n}(S_{i,j}^n)=m_{i,l}.
$$
Consider the map $\Phi_{i,l}^n=\Phi_n \circ (\Psi_{i,l}^n)^{-1}\colon (D_{i,j}^n,S_{i,j}^n) \to (\mathbb B^{N},\mathbb S^{N-1})$. We endow $D_{i,j}^n$ with the metric $(\Psi_{i,l}^n)_*g_n$ and $\mathbb B^{N}$ with the Euclidean metric. Then the map $\Phi_{i,l}^n$ is harmonic with free boundary since $\Phi_n$ is harmonic with free boundary and $\Psi_{i,l}^n$ is conformal. Moreover, it is shown in \cite{petrides2019maximizing} that the measure $\boldsymbol{1}_{S_{i,j}^n}\langle\Phi_{i,l}^n,\partial_\nu\Phi_{i,l}^n\rangle_{g_{can}}ds_{g_{can}}$ does not concentrate at the poles $(0,1)$ and $(0,-1)$ of $\mathbb{D}^2$. Indeed, if the measure concentrated at the poles then one would obtain a contradiction with the maximality of $s_1+\sum^{s_1+s_2}_{i=s_1+1}t_i$.

The exactly same procedure can be carried out for components $\Sigma_j^n(\alpha^n)$, $j\in J$. The only difference is that now we use restrictions of diffeomorphisms $\Psi^n$ given by Proposition~\ref{D-M} instead of the explicit harmonic map as above. As a result, one obtains domains $\check\Omega^n_j\subset \Sigma_\infty$ and harmonic maps with free boundary $\check\Phi^n_j\colon\check\Omega^n_j\to\mathbb{B}^N$ such that the measure $\boldsymbol{1}_{\partial\check\Omega_j^n}\langle\Phi_{i,l}^n,\partial_\nu\Phi_{i,l}^n\rangle_{g_{can}}ds_{g_{can}}$ does not concentrate at the marked points of $\widehat{\Sigma_\infty}$.


Now thanks to inequality \eqref{gap} we can construct $k+1$ well-defined test-functions for the Rayleigh quotient of $\sigma_k$ using the limit functions of the sequences of maps $\hat \Phi_{i,l}^n$ and $\hat \Phi_i^n$ as it was shown in \cite{petrides2019maximizing}. Precisely, let $p_{i}$ be the maximal integers such that
\begin{gather}\label{tilde i}
\frac{\sigma^*_{p_{i}}(\mathbb{D}^2)}{m_{i}}<\limsup_{n\to\infty}\sigma^*_k(\Sigma,c_n),
\end{gather}
where $1 \leq i \leq s_1$, $p_{i,l}$ the maximal integers such that
\begin{align}\label{tilde i,l}
\frac{\sigma^*_{p_{i,l}}(\mathbb{D}^2)}{m_{i,l}}<\limsup_{n\to\infty}\sigma^*_k(\Sigma,c_n),
\end{align}
where $s_1+1 \leq i \leq s_1+s_2$ and $p_j$ the maximal integers such that
\begin{align}\label{j}
\frac{\sigma^*_{p_{j}}(\widehat{\Sigma_j^\infty}, \widehat{c_{\infty}})}{m_{j}}<\limsup_{n\to\infty}\sigma^*_k(\Sigma,c_n),~j\in J.
\end{align}
Then one has
$$
\sigma^*_{p_{i}+1}(\mathbb{D}^2) \geq m_{i}\limsup_{n\to\infty}\sigma^*_k(\Sigma,c_n),~1 \leq i \leq s_1,
$$
$$
\sigma^*_{p_{i,l}+1}(\mathbb{D}^2) \geq m_{i,l}\limsup_{n\to\infty}\sigma^*_k(\Sigma,c_n),~s_1+1 \leq i \leq s_1+s_2
$$and
$$
\sigma^*_{p_{j}+1}(\widehat{\Sigma_j^\infty}, \widehat{c_{\infty}}) \geq m_{j}\limsup_{n\to\infty}\sigma^*_k(\Sigma, c_n),~j\in J.
$$
If $\sum^{s_1}_{i=1} (p_{i}+1)+\sum^{s_1+s_2}_{i=s_1+1} \sum^{t_i}_{l=1} (p_{i,l}+1)+\sum_{j \in J}(p_j+1) \leq k$ then by inequality \eqref{gap} we have
\begin{gather*}
\sum^{s_1}_{i=1} \sigma^*_{p_{i}+1}(\mathbb{D}^2)+\sum^{s_1+s_2}_{i=s_1+1} \sum^{t_i}_{l=1}\sigma^*_{p_{i,l}+1}(\mathbb{D}^2)+ \sum_{j \in J}\sigma^*_{p_{j}+1}(\widehat{\Sigma_j^\infty}, \widehat{c_{\infty}})<\limsup_{n\to\infty}\sigma^*_k(\Sigma,c_n),
\end{gather*}
which implies $\sum^{s_1}_{i=1} m_{i}+\sum^{s_1+s_2}_{i=s_1+1} \sum^{t_i}_{l=1}m_{i,l}+\sum_{j\in J}m_j < 1$ and we arrive at a contradiction with Proposition \ref{sequences}. Hence, $\sum^{s_1}_{i=1} (p_{i}+1)+\sum^{s_1+s_2}_{i=s_1+1} \sum^{t_i}_{l=1} (p_{i,l}+1)+\sum_{j \in J}(p_j+1) \geq k+1$. 

Further, let $dv_{g^{i}_\infty}=\lim_{n \to \infty} (\Psi_{i}^n)_*dv_{g_n}$, $dv_{g^{i,l}_\infty}=\lim_{n \to \infty} (\Psi_{i,l}^n)_*dv_{g_n}$ and $dv_{g^j_\infty}=\lim_{n \to \infty} (\Psi_j^n)^*dv_{g_n}$. Denote by $\widehat{dv_{g^{i}_\infty}}$, $\widehat{dv_{g^{i,l}_\infty}}$ and $\widehat{dv_{g^j_\infty}}$ the measures induced by the compactification  on $\mathbb{D}^2$ for $1 \leq i \leq s_1$ and $s_1+1 \leq i \leq s_1+s_2$ and on $\widehat{\Sigma_j^\infty}$ respectively. These measures are well-defined due to the non-concentration argument explained above. Take orthonormal families of eigenfucntions $(\phi^0_i,\ldots,\phi^{p_{i}}_i)$ in $L^2(\mathbb{D}^2, \widehat{dv_{g^{i}_\infty}})$ $1 \leq i \leq s_1$, $(\phi^0_i,\ldots,\phi^{p_{i,l}}_i)$ in $L^2(\mathbb{D}^2, \widehat{dv_{g^{i,l}_\infty}})$ $s_1+1 \leq i \leq s_1+s_2$ and $(\psi^0_j,\ldots,\psi^{p_{j}}_j)$ in $L^2(\widehat{\Sigma_j^\infty}, \widehat{dv_{g^j_\infty}})$ such that for $0 \leq e \leq p_{i}$ the function $\phi^e_i$ is an eigenfunction with eigenvalue $\sigma_e(\widehat{dv_{g^{i}_\infty}})$ on $\mathbb{D}^2$, for $0 \leq e \leq p_{i,l}$ the function $\phi^e_i$ is an eigenfunction with eigenvalue $\sigma_e(\widehat{dv_{g^{i,l}_\infty}})$ on $\mathbb{D}^2$ and for $0 \leq r \leq p_{j}$ the function $\psi^r_j$ is an eigenfunction with eigenvalue $\sigma_r(\widehat{dv_{g^j_\infty}})$ on $\widehat{\Sigma_j^\infty}$. The standard capacity computations (see for instance \cite[Claim 1]{petrides2019maximizing}) imply the existence of smooth functions supported in a geodesic ball of a Riemannian manifold and having bounded Dirichlet energy. More precisely there exist positive smooth functions $\eta_i$, $\eta_{i,l}$ and $\eta_j$ for $(\mathbb{D}^2, \widehat{dv_{g^{i}_\infty}})$, $(\mathbb{D}^2, \widehat{dv_{g^{i,l}_\infty}})$ and $(\widehat{\Sigma_j^\infty}, \widehat{dv_{g^j_\infty}})$ respectively supported in geodesic balls $B(x,r)$ centered at the compactification points $x$ of radius $r$ such that $\eta \in C^\infty_0(B(x,r))$ and $\eta =1$ on $B(x,\rho_n r)\subset B(x,r)$ where $\rho_n\to 0$ as $n\to\infty$ and $\int_\Omega|\nabla \eta|^2_gdv_g \leq \frac{C}{\log\frac{1}{\rho_n}},$ where $\eta$ is one of the functions $\eta_i$, $\eta_{i,l}$ and $\eta_j$, $(\Omega,dv_g)$ is one of the corresponding manifolds $(\mathbb{D}^2, \widehat{dv_{g^{i}_\infty}})$, $(\mathbb{D}^2, \widehat{dv_{g^{i,l}_\infty}})$ and $(\widehat{\Sigma_j^\infty}, \widehat{dv_{g^j_\infty}})$. Moreover, if $(\Omega,dv_g)=(\mathbb{D}^2, \widehat{dv_{g^{i,l}_\infty}})$ then we additionally require $\rho_n$ to satisfy $\partial D^n_{i,l}\setminus S^n_{i,l} \subset B(x,\rho_n r)$. 
 Then we define the desired test-functions as
$$
\xi^e_i=(\Psi_{i}^n)^{-1}\eta_i\phi^e_i, ~1 \leq i \leq s_1
$$
extended by 0 on $\Sigma$,
$$
\xi^e_{i,l}=(\Psi_{i,l}^n)^{-1}\eta_{i,l}\phi^e_i, ~s_1+1 \leq i \leq s_1+s_2
$$
extended by 0 on $\Sigma$ and
$$
\xi^r_j=\Psi_j^n\eta_j\psi^r_j,~j\in J
$$
extended by 0 on $\Sigma$. Note that all these functions have pairwise disjoint supports. Then from the variational characterization of $\sigma_k(g_n)$ one gets
\begin{gather*}
\sigma_k(g_n) \leq \max \Big\{\max_{1 \leq i \leq s_1} \frac{\int_{\Sigma}|\nabla \xi^e_i|^2_{g_n}dv_{g_n}}{\int_{\partial \Sigma} (\xi^e_i)^2 ds_{g_n}}, \max_{s_1+1 \leq i \leq s_1+s_2} \frac{\int_{\Sigma}|\nabla \xi^e_{i,l}|^2_{g_n}dv_{g_n}}{\int_{\partial \Sigma} (\xi^e_{i,l})^2 ds_{g_n}},\\ \max_{j \in J} \frac{\int_{\Sigma}|\nabla \xi^r_j|^2_{g_n}dv_{g_n}}{\int_{\partial\Sigma} (\xi^r_j)^2 ds_{g_n}} \Big\},
\end{gather*}
and passing to $\limsup$ as $n\to\infty$ we get
\begin{gather*}
\limsup_{n\to\infty}\sigma^*_k(\Sigma,c_n) \leq \max\Big\{\max_{1 \leq i \leq s_1} \frac{\sigma^*_{p_{i}}(\mathbb{D}^2)}{m_{i}}, \max_{s_1+1 \leq i \leq s_1+s_2} \frac{\sigma^*_{p_{i,l}}(\mathbb{D}^2)}{m_{i,l}},\\ \max_{j \in J}\frac{\sigma^*_{p_{j}}(\widehat{\Sigma_j^\infty}, \widehat{c_{\infty}})}{m_{j}}\Big\}
\end{gather*}
which contradicts \eqref{tilde i}, \eqref{tilde i,l} and \eqref{j}. This means that if inequality \eqref{gap} holds then the sequence $\{c_n\}$ cannot degenerate. We arrived at a contradiction and inequality~\eqref{aim'} is proved. 

\begin{remark}\label{useful}
Note that if $s_2=0$, i.e. there are no pinching geodesics having intersection with boundary components, then we take the set $J$ as $J=\{1,\ldots,m\}$, i.e. we consider $\Sigma^n_j(\alpha^n)$ where $1 \leq j \leq m$. If all the boundary components are getting pinched then we set $J=\O$ and we only have deal with the functions $
\xi^e_i=(\Psi_{i}^n)^{-1}\eta_i\phi^e_i$ extended by 0 on $\Sigma$ and $\sigma^*_{p_{i}}(\mathbb{D}^2)$ where $1 \leq i \leq s_1$. If $s_1=s_2=0$, i.e. only geodesics of the third type are getting pinched then we only have deal with functions $\xi^r_j=\Psi_j^n\eta_j\psi^r_j,~j\in J$ extended by 0 on $\Sigma$ and $\sigma^*_{p_{j}}(\widehat{\Sigma_j^\infty}, \widehat{c_{\infty}})$ where $J=\{1,\ldots,m\}$. 
\end{remark}

{\bf Case 2.} Assume that up to a choice of  a subsequence the following inequality holds
\begin{align*}
\sigma^*_k(\Sigma, c_n) \leq \sigma^*_{k-1}(\Sigma, c_n) +2\pi
\end{align*}
then we prove inequality \eqref{aim'} by induction.

Consider the case $k=1$ then by inequality~ \eqref{el soufi} $\sigma^*_1(\Sigma, c_n) \geq 2\pi$. Suppose that up to a choice of a subsequence one has $\sigma^*_1(\Sigma, c_n) > 2\pi$. Then the case $k=1$ falls under Case 1. Otherwise one has $\limsup_{n \to \infty} \sigma^*_1(\Sigma, c_n)=2\pi$ and the inequality ~\eqref{aim'} reads as
\begin{gather*}
2\pi=\limsup_{n \to \infty} \sigma^*_1(\Sigma, c_n) \leq \max \{\sigma^*_{1}(\Sigma_{\gamma_{i},l_i},c_\infty);2\pi\},
\end{gather*}
which is true. The base of induction is proved.

Suppose that the inequality holds for all numbers $k'\leq k$. We show that it also holds for $k+1$. Indeed, one has
\begin{align*}
\sigma^*_{k+1}(\Sigma, c_n) \leq \sigma^*_{k}(\Sigma, c_n)+2\pi=\sigma^*_{k}(\Sigma, c_n)+\sigma^*_1(\mathbb{D}^2)
\end{align*}
and we get 

\begin{gather*}
\limsup_{n \to \infty}\sigma^*_{k+1} (\Sigma, c_n) \leq \max \Big(\sum^{m}_{i=1} \sigma^*_{k_{i}}(\Sigma_{\gamma_{i},l_i},c_\infty)+ \sum_{i=1}^{s_1+s_2}\sigma^*_{r_i}(\mathbb{D}^2) \Big)+\sigma^*_1(\mathbb{D}^2) \leq \\ \leq \max \Big(\sum^{m}_{i=1} \sigma^*_{k_{i}}(\Sigma_{\gamma_{i},l_i},c_\infty)+ \sum_{i=1}^{s_1+s_2}\sigma^*_{r_i}(\mathbb{D}^2) \Big),
\end{gather*}
where the maximum is taken over all possible combinations of indices such that
 $$
 \sum_{i=1}^{m} k_i + \sum_{i=1}^{s_1+s_2} r_i = k+1,
 $$
 since the term $\sigma^*_1(\mathbb{D}^2)$ can be absorbed by one of the terms inside $\max$ using inequality~\eqref{petya}. The proof is complete. 
 
{\bf Zero Euler characteristic.} The case of the cylinder was essentially considered in~\cite[Section 7.1]{petrides2019maximizing}. Indeed, it was proved that if the sequence of conformal classes $\{c_n\}$ degenerates then
$$
\lim_{n\to\infty}\sigma^*_k(\mathcal C,c_n) \leq \max_{i_1+\cdots+i_s=k}\sum^s_{q=1}\sigma^*_{i_q}(\mathbb D^2)=2\pi k.
$$
Applying then inequality~\eqref{el soufi} one immediately gets that $\lim_{n\to\infty}\sigma^*_k(\mathcal C,c_n)=2\pi k$.

Consider the case of the M\"obius band. If the sequence $\{c_n\}$ goes to $0$ then it follows from~\cite[Section 7.1]{petrides2019maximizing} that
\begin{gather}\label{MBcase}
\lim_{n\to\infty}\sigma^*_k(\mathbb{MB},c_n) \leq \max_{i_1+\cdots+i_s=k}\sum^s_{q=1}\sigma^*_{i_q}(\mathbb D^2)=2\pi k.
\end{gather}
Indeed, we pass to the orientable cover which is a cylinder. Then inequality~\eqref{MBcase} follows from~\cite[Section 7.1, the case $R_\alpha \to 1$ as $\alpha \to+\infty$ in Petrides' notations]{petrides2019maximizing}.

If the sequence $\{c_n\}$ goes to $\infty$ then we prove that inequality~\eqref{MBcase} also holds. The proof follows the exactly same arguments as in the proof of inequality~\eqref{aim'}. The analog of the case 1 for $\mathbb{MB}$ corresponds to the case of pinching boundary (see Remark~\eqref{useful}).

Therefore, in both cases inequality~\eqref{MBcase} holds. Applying inequality~ \eqref{el soufi} once again we then get that $\lim_{n\to\infty}\sigma^*_k(\mathbb{MB},c_n)=2\pi k$.

\medskip

\section{Proof of Theorem \ref{disproof}}
\label{main theorem proof}
For the proof of Theorem \ref{disproof} we will need to choose a "nice" degenerating sequence of conformal classes, i.e. a degenerating sequence of conformal classes such that the limiting space looks as simple as possible. 

\begin{lemma}\label{metric sequences}
Let $\Sigma$ be a compact surface with boundary of negative Euler characteristic. Then there exists a degenerating sequence of conformal classes such that the limiting space is the disc.



\end{lemma}

\begin{proof}

The proof is purely topological. 

Assume that $\Sigma$ is orientable. Then we consider collapsing geodesics shown in Figure~\ref{F1}. Passing to the limit when the lengths of all pinching geodesics tend to zero and using the one-point cusps compactification we get an orientable surface of genus 0 with one boundary component, i.e. the disc. 


\begin{figure}[h!]
  \centering
  \def\svgwidth{\columnwidth}
  \includegraphics[scale=0.5]{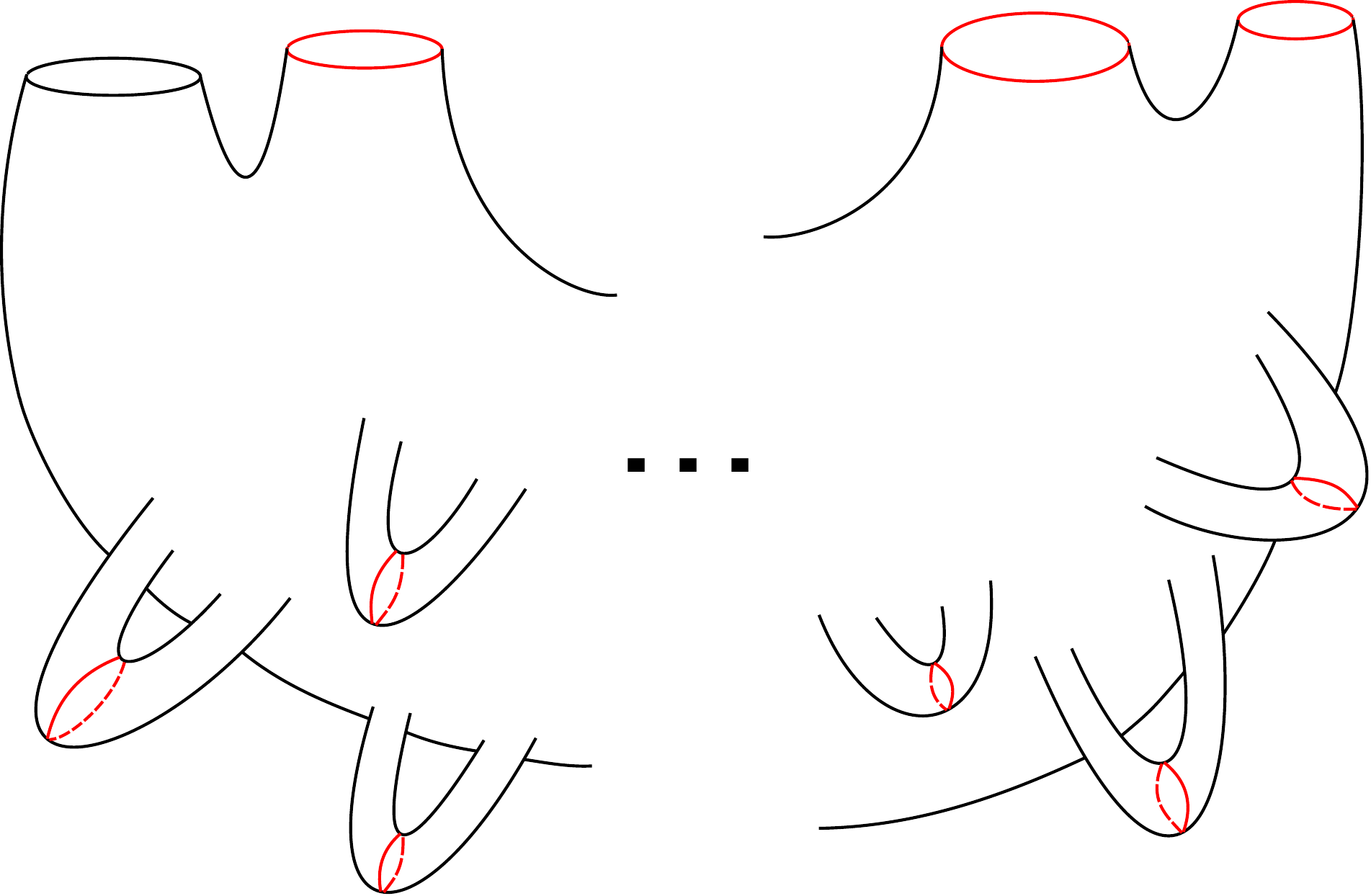}
    \footnotesize
    \caption{
    Orientable surface with boundary. The lengths of all \textit{red} geodesics tend to zero.
    }
     \label{F1}
\end{figure}

If $\Sigma$ is non-orientable then we pass to its orientable cover and we consider collapsing geodesics shown in Figure~\ref{F2} for genus $0$ and Figure~\ref{F3} for genus $\neq 0$ (the pictures are symmetric with respect to the involution changing the orientation, "the antipodal map"). Passing to the limit when the lengths of all pinching geodesics tend to zero and using the one-point cusps compactification we get a disconnected surface with two connected components which are topologically discs. The involution changing the orientation maps one component to another one and hence passing to the quotient by this involution we get just one disc. 


\begin{figure}[h!]
  \centering
  \def\svgwidth{\columnwidth}
  \includegraphics[scale=0.9]{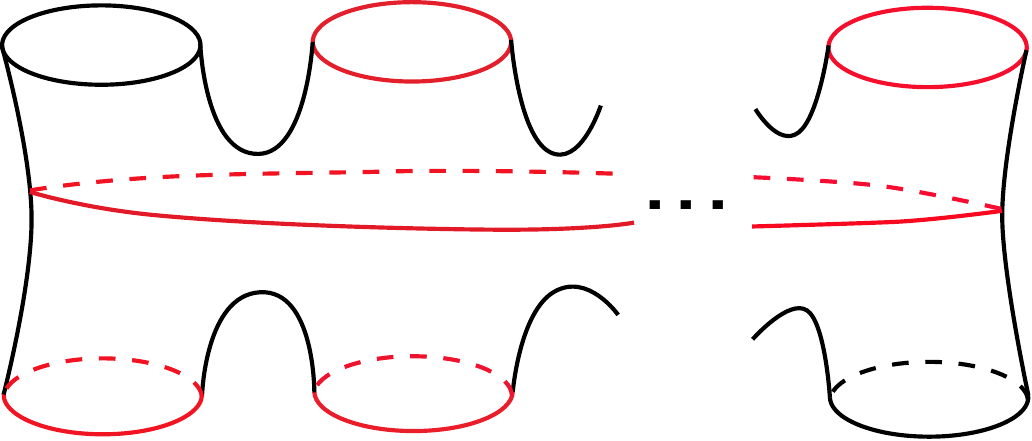}
    \footnotesize
    \caption{
    Orientable cover of a non-orientable surface of genus $0$ with boundary. The lengths of all \textit{red} geodesics tend to zero. 
    }
    \label{F2}
\end{figure}

\begin{figure}[h!]
  \centering
  \def\svgwidth{\columnwidth}
  \includegraphics[scale=0.8]{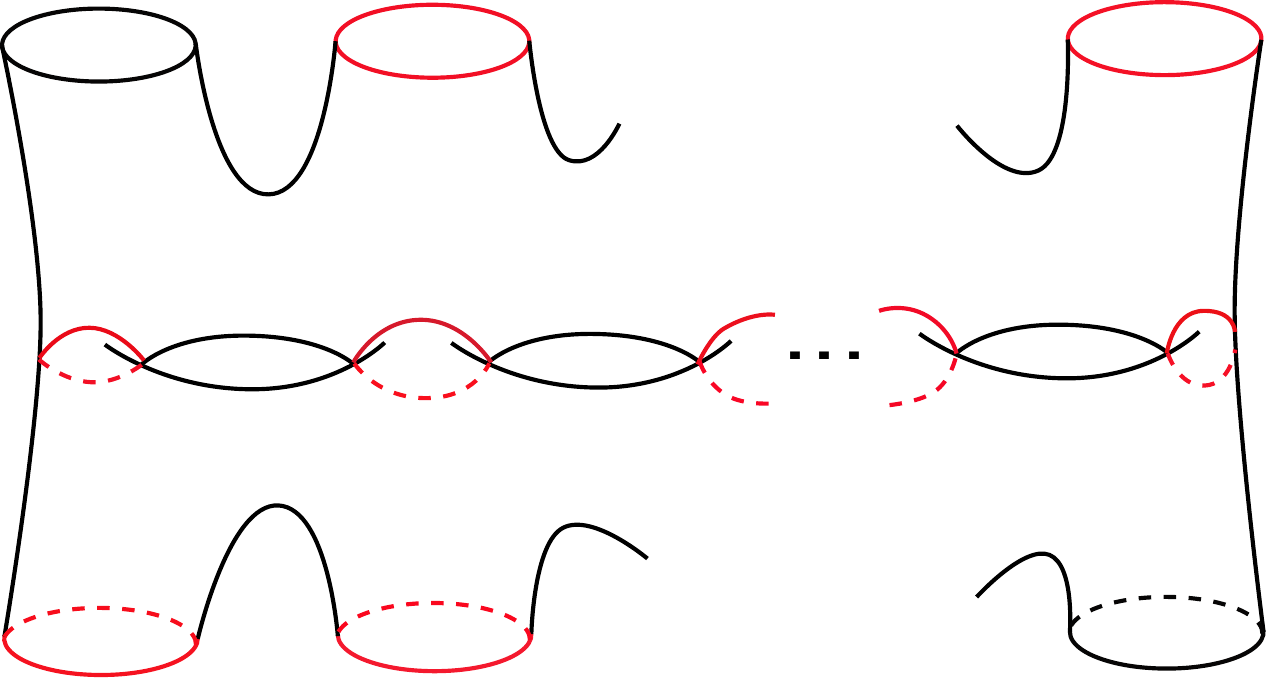}
    \footnotesize
    \caption{
     Orientable cover of a non-orientable surface of genus $\neq 0$ with boundary. The lengths of all \textit{red} geodesics tend to zero. 
      }
      \label{F3}
\end{figure}



\end{proof}


Now we are ready to prove Theorem \ref{disproof}.

\medskip

{\bf Zero Euler characteristic.} Let $\Sigma$ be either the cylinder $\mathcal C$ or the M\"obius band $\mathbb{MB}$.  Then this case immediately follows from Theorem~\ref{conf&conv} by Remark~\ref{main_remark}. Indeed, if $\{c_n\}$ denotes a degenerating sequence of conformal classes on $\Sigma$ then by Theorem~\ref{conf&conv}:
$$
I^\sigma_k(\Sigma) \leq \lim_{n\to\infty}\sigma^*_k(\Sigma,c_n)=2\pi k.
$$
But $I^\sigma_k(\Sigma) \geq 2\pi k$ by \eqref{el soufi}. Thus $I^\sigma_k(\Sigma)=\lim_{n\to\infty}\sigma^*_k(\Sigma,c_n)=2\pi k$ and the degenerating sequence $\{c_n\}$ is minimizing. 

\medskip

{\bf Negative Euler characteristic.}
By Lemma~\ref{metric sequences} there exists a sequence of conformal classes $\{c_n\}$ such that the limiting space $\widehat{\Sigma_\infty}$ is the disc. Then by Theorem~\ref{conf&conv} we have
$$
\lim_{n\to\infty}\sigma^*_k(\Sigma,c_n) = \max_{\sum k_j=k}\sum\sigma^*_{k_j}(\mathbb{D}^2).
$$
Moreover, we know that $\sigma^*_k(\mathbb{D}^2) = 2\pi k$. Hence,
$$
I^\sigma_k(\Sigma) \leq \lim_{n\to\infty}\sigma^*_k(\Sigma,c_n) = 2\pi k.
$$
Finally, by~\eqref{el soufi} one has $I^\sigma_k(\Sigma) \geq 2\pi k$ whence $I^\sigma_k(\Sigma)=2\pi k$ which completes the proof.

\section{Appendix} 
\label{appendix}

\subsection{A well-posed problem.} \label{appendix1}

In this section we consider the problem
\begin{gather}
\label{toprove}
\begin{cases}
\Delta u=0&\text{in $M$},\\
u=g&\text{on $D$},\\
\frac{\partial u}{\partial n}=0&\text{on $N$},
\end{cases}
\end{gather}
where $(M,h)$ is a Riemannian manifold with boundary such that $\overline D\cup \overline N=\partial M$ and $D$ has positive capacity.

Let $G$ be a smooth function such that $G_{|_D}=g$ and consider the function $v=G-u$. Then substituting $u=G-v$ into \eqref{toprove} implies: 
\begin{gather}
\label{toprove1}
\begin{cases}
\Delta v=\Delta G&\text{in $M$},\\
v=0&\text{on $D$},\\
\frac{\partial u}{\partial n}=\frac{\partial G}{\partial n}&\text{on $N$}.
\end{cases}
\end{gather}

We introduce the space $H^1_D(M,h)$ as the closure in $H^1-$norm of $C^\infty-$functions vanishing on $D$. For a function $u \in H^1_D(M,h)$ we have the following coercivity inequality:
\begin{gather}\label{coercivity}
||u||_{L^2(M,h)} \leq C||\nabla u||_{L^2(M,h)},
\end{gather}
with the best constant $C=\frac{1}{\sqrt{\lambda^{DN}_1(M,h)}}$, where $\lambda^{DN}_1(M,h)$ is the first non zero eigenvalue of the mixed problem
\begin{gather}
\begin{cases}
\Delta u=\lambda u&\text{in $M$},\\
u=0&\text{on $D$},\\
\frac{\partial u}{\partial n}=0&\text{on $N$}.
\end{cases}
\end{gather}
By the Lax-Milgram theorem and by virtue of the inequality \eqref{coercivity} the problem \eqref{toprove1} admits a unique solution on the space $H^1_D(M,h)$. Thus, problem \eqref{toprove} also has a solution. Moreover, it is easy to see that this solution is unique. 

Our aim now is the following lemma.

\begin{lemma} \label{finally}
Let $u$ satisfy the problem \eqref{toprove}.
Then one has
\begin{gather*}
||u||_{H^1(M,h)} \leq C||g||_{H^{1/2}(D,h)}.
\end{gather*}
\end{lemma}

\begin{proof}
The weak formulation of \eqref{toprove} reads
$$
\int_M\langle\nabla u, \nabla v\rangle dv_h=0,~\forall v\in H^1_D(M,h).
$$
Let $G$ be any continuation of the function $g$ into $M$, i.e. $G\in H^1(M,h)$ is any function such that $G_{|_D}=g$. Then substituting $v=u-G$ in the previous identity yields 
\begin{gather*}
0=\int_M\langle\nabla u, \nabla u-\nabla G\rangle dv_h=\int_M|\nabla u|^2dv_h-\int_M\langle\nabla u, \nabla G\rangle dv_h,
\end{gather*}
whence 
\begin{gather}
\label{I1}
\int_M|\nabla u|^2dv_h=\int_M\langle\nabla u, \nabla G\rangle dv_h \leq \frac{1}{2}\int_M|\nabla u|^2dv_h+\frac{1}{2}\int_M|\nabla G|^2dv_h.
\end{gather}
Further, it is easy to see that
$$||u||_{L^2(M,h)} \leq ||u-G||_{L^2(M,h)}+||G||_{L^2(M,h)}.$$
Moreover, since $u-G\in H^1_D(M,h)$ one has
$$
||u-G||_{L^2(M,h)} \leq C||\nabla u -\nabla G||_{L^2(M,h)} \leq C(||\nabla u||_{L^2(M,h)}+||\nabla G||_{L^2(M,h)}).
$$
Substituting it in the previous inequality we get
\begin{gather}
\label{I2}
||u||_{L^2(M,h)} \leq C(||\nabla u||_{L^2(M,h)}+||\nabla G||_{L^2(M,h)})+||G||_{L^2(M,h)}.
\end{gather}
Plugging \eqref{I1} in \eqref{I2} yields 
\begin{gather}
\label{I3}
||u||_{L^2(M,h)} \leq C||G||_{H^1(M,h)}.
\end{gather}
Finally \eqref{I1} and \eqref{I3} imply
\begin{gather}
\label{l4}
||u||_{H^1(M,h)} \leq C||G||_{H^1(M,h)}
\end{gather}
for any function $G\in H^1(M,h)$ such that $G_{|_D}=g$. 
\begin{lemma}\label{inf}
The norms $$\inf_{G\in H^1(M,h),~G_{|_D}=g}||G||_{H^1(M,h)}~\text{and}~||g||_{H^{1/2}(D,h)}$$ are equivalent.
\end{lemma}
\begin{proof}
By the trace inequality there exists a positive constant $C_1$ such that for every $G\in H^1(M,h)$ one has
\begin{gather*}
||g||_{H^{1/2}(D,h)} \leq C_1||G||_{H^1(M,h)},
\end{gather*}
which implies:
\begin{gather}
\label{l5}
||g||_{H^{1/2}(D,h)} \leq C_1\inf_{G\in H^1(M,h),~G_{|_D}=g}||G||_{H^1(M,h)};
\end{gather}
Further, we construct a continuation $G'\in H^1(M,h)$ of $g$ with the property that there exists a positive constant $C_2$ such that for every $g\in H^{1/2}(D,h)$ one has: 
\begin{gather}
\label{l6}
||G'||_{H^1(M,h)} \leq C_2||g||_{H^{1/2}(D,h)}.
\end{gather}
Let $\tilde g$ be any continuation of $g$ on $\partial M$ such that $||\tilde g||_{H^{1/2}(N,h)} \leq ||g||_{H^{1/2}(D,h)}$. Therefore, $||\tilde g||_{H^{1/2}(\partial M,h)} \leq \sqrt{2}||g||_{H^{1/2}(D,h)}<\infty$ and $\tilde g\in H^{1/2}(\partial M,h)$. Then we take the harmonic continuation of $\tilde g$ into $M$ as $G'$. By \cite[Proposition 1.7]{Taylor} there exists a positive constant that $C_3$ such that:
$$
||G'||_{H^1(M,h)} \leq C_3||\tilde g||_{H^{1/2}(\partial M,h)}.
$$
Since $||\tilde g||_{H^{1/2}(\partial M,h)} \leq \sqrt{2}||g||_{H^{1/2}(D,h)}$ we get~\eqref{l6} with $C_2=\sqrt{2}C_3$.

Therefore, \eqref{l5} and \eqref{l6} imply:
$$
C_2^{-1}||G'||_{H^1(M,h)} \leq ||g||_{H^{1/2}(D,h)} \leq C_1\inf_{G\in H^1(M,h),~G_{|_D}=g}||G||_{H^1(M,h)},
$$
whence
\begin{gather*}
C_2^{-1}\inf_{G\in H^1(M,h),~G_{|_D}=g}||G||_{H^1(M,h)} \leq ||g||_{H^{1/2}(D,h)} \leq \\ \leq C_1\inf_{G\in H^1(M,h),~G_{|_D}=g}||G||_{H^1(M,h)},
\end{gather*}
since $$||G'||_{H^1(M,h)} \geq\inf_{G\in H^1(M,h),~G_{|_D}=g}||G||_{H^1(M,h)}.$$ And lemma follows.
\end{proof}
Finally, taking the infimum over all $G\in H^1(M,h)$ such that $G_{|_D}=g$ in~\eqref{l4} and using Lemma~\ref{inf} complete the proof.
\end{proof}

\subsection{Proofs of propositions of Section \ref{analysis}.}\label{appendix2}
This section contains the proofs of propositions in section \ref{analysis} analogous to propositions  in \cite[Section 4]{karpukhin2019friedlander} whose adaptation to the Steklov setting is rather technical. 

\begin{proof}[Proof of Lemma \ref{identity2}]
Let $h^m\in[h]$ be a maximizing sequence of metrics for $\sigma^{N*}_k(\Omega, \partial^S\Omega, [h])$ and $g^m\in[g]$ be a discontinuous metric on $\Sigma$ defined as $g|_{\Omega_i} = h_i$. By the variational characterization of eigenvalues for all $k$ one has $\sigma_k(\Sigma,g^m) \geq\sigma^N(\Omega,h^m)$ since the set of test functions for the Steklov-Neumann eigenvalues $C^0(\Sigma,\{\Omega_i\})$ is larger than the set $C^0(\Sigma)$ of test functions for $\sigma_k(\Sigma,g^m)$. Using the fact that $L_{g^m}(\partial \Sigma)=\sum_iL_{h^m}(\partial^S \Omega_i) \geq L_{g^m}(\partial^S \Omega_i)$ for any $i$ and taking the limit as $m\to\infty$ we get
$$
\sigma^*_k(\Sigma,\{\Omega_i\},[g]) \geq\sigma_k^{N*}(\Omega, \partial^S\Omega, [h]).
$$
Finally by Lemma~\ref{identity} one gets
$$
\sigma^*_k(\Sigma,[g]) \geq \sigma^{N*}_k(\Omega,\partial^S\Omega, [h]).
$$
\end{proof}

\begin{proof}[Proof of Proposition \ref{subdomain}]
The proof is similar for both cases. The obvious analog of Lemma~\ref{liminf} for the second case holds since its proof follows the exactly same arguments as the proof of Lemma~\ref{liminf}. For that reason we only provide the proof of Proposition \ref{subdomain} for the first case. 

Take a maximizing sequence of metrics $\{h_i~ |~ h_i \in [g|_{\Omega}]\}$ for the functional  $\sigma^{N*}_k(\Omega, \partial^S\Omega, [g])$, i.e. 

\begin{align*}
\lim_{i\to\infty}\bar{\sigma}^{N}_k(\Omega,\partial^S\Omega, h_i)=\sigma^{N*}_k(\Omega,\partial^S\Omega, [g])
\end{align*}

Let $h_i=f_i g|_{\Omega}$, where $f_i \in C^\infty_+(\bar\Omega)$. We then define the metric $\widetilde{h_i}=\widetilde{f_i} g$ on $\Sigma$, where $\widetilde{f_i}$ is any positive continuation of the function $f_i$ into $\Omega^c$. It enables us to consider the metric $\rho_\delta \widetilde{h_i}$, where as before
 \begin{align*}
\rho_\delta=
 \begin{cases}
 1&\text{in $\Omega$},\\
 \delta&\text{in $\Sigma\setminus\Omega$}.
 \end{cases}
\end{align*}
Lemma \ref{liminf} implies

\begin{align*}
\liminf_{\delta \to 0} \sigma_k(\rho_\delta \widetilde{h_i}) \geq \sigma^N_k(\Omega,\partial^S\Omega, h_i). 
\end{align*}
Moreover, $L_{\rho_\delta\widetilde{h_i}}(\partial \Sigma)\to L_{h_i}(\partial^S\Omega)$.
By Lemma \ref{identity} we have

\begin{align*}
\sigma^*_k(\Sigma, [g])=\sigma^*_k(\Sigma,\{\Omega,\Sigma\setminus\Omega\},[g]) \geq \liminf_{\delta\to 0}\bar\sigma_k(\rho_\delta \widetilde{h_i}) \geq\bar\sigma^N_k(\Omega,\partial^S\Omega, h_i).
\end{align*}
Therefore, passing to the limit as $i \to \infty$ one gets,

\begin{align*}
\sigma^*_k(\Sigma, [g]) \geq \sigma^{N*}_k(\Omega,\partial^S\Omega, [g]).
\end{align*}

\end{proof}

\begin{proof}[Proof of Corollary \ref{Neumann cor2}]
We show that 
$$
\sigma^*_k(M, [g]) \leq \liminf_{n \to \infty}\sigma^{N*}_k(M \setminus K_n, \partial M \setminus \partial K_n, [g]).
$$
Let $g^m$ be a maximizing sequence for the functional $\sigma^*_k(M, [g])$. For a fixed $m$ we consider geodesic balls $B_{\epsilon_n}(p_i)$ of radius $\epsilon_n\to 0$ in metric $g^m$ centred at the points $p_1,\ldots,p_l \in M$ such that $K_n \subset \cup^l_{i=1}B_{\epsilon_n}(p_i)$. We see that $M\setminus \cup^l_{i=1}B_{\epsilon_n}(p_i) \subset M\setminus K_n$. Then by Proposition~\ref{subdomain} one has
\begin{gather}\label{420}
\sigma^{N*}_k(M \setminus K_n, \partial M \setminus \partial K_n, [g]) \geq \\ \geq\sigma^{N*}_k(M\setminus  \cup^l_{i=1}B_{\epsilon_n}(p_i), \partial M\setminus \cup^l_{i=1}\partial B_{\epsilon_n}(p_i), [g]) \geq \\ \geq
\bar\sigma^{N}_k(M\setminus  \cup^l_{i=1}B_{\epsilon_n}(p_i), \partial M\setminus \cup^l_{i=1}\partial B_{\epsilon_n}(p_i), g^m).
\end{gather}
Note that $L(\partial M \setminus \cup^l_{i=1}\partial B_{\epsilon_n}(p_i), g^m)\to L(\partial M, g^m)$ as $n\to\infty$ and by Lemma~\ref{Neumann conv} one has $\sigma^{N}_k(M\setminus  \cup^l_{i=1}B_{\epsilon_n}(p_i), \partial M\setminus \cup^l_{i=1}\partial B_{\epsilon_n}(p_i), g^m)\to\sigma_k(M,g^m)$. Hence, $\bar\sigma^{N}_k(M\setminus  \cup^l_{i=1}B_{\epsilon_n}(p_i), \partial M\setminus \cup^l_{i=1}\partial B_{\epsilon_n}(p_i), g^m) \to \bar\sigma_k(M,g^m)$ as $n\to\infty$. 
Taking $\liminf_{n\to\infty}$ in~\eqref{420} one then gets
$$
\liminf_{n\to\infty}\sigma^{N*}_k(M \setminus K_n, \partial M \setminus \partial K_n, [g]) \geq \bar\sigma_k(M,g^m).
$$
Passing to the limit as $m\to\infty$ we get the desired inequality.

The inequality
$$
\limsup_{n \to \infty}\sigma^{N*}_k(M \setminus K_n, \partial M \setminus \partial K_n, [g]) \leq \sigma^*_k(M, [g])
$$
follows from Proposition~\ref{subdomain}. This completes the proof.
\end{proof}

\begin{proof}[Proof of Lemma~\ref{disconnected}]
Essentially the idea of the proof comes from the paper \cite{WK94}. We denote by $\partial^S\Omega$ the part of the boundary with the Steklov boundary condition. We also call $\partial^S\Omega$ "Steklov boundary" and $L_g(\partial^S\Omega)$ "the length of Steklov boundary" in metric $g$. 

{\bf Inequality $\geq$.} 

Fix the indices $k_i>0$ satisfying $\sum k_i=k$ and consider a maximizing sequence of metrics $\{g_i^m\}$ such that  $\bar\sigma^N_{k_i}(\Omega_i,\partial^S\Omega_i, g^m_i)\to\sigma^{N*}_{k_i}(\Omega_i,\partial^S\Omega_i, [g_i])$. One can assume that $\sigma^N_{k_i}(\Omega_i,\partial^S\Omega_i, g^m_i)=\sigma^{N*}_k(\Omega,\partial^S\Omega, [g])$. Then, one has
$$
L_{g^m_i}(\partial^S\Omega_i)\to \frac{\sigma^{N*}_{k_i}(\Omega_i,\partial^S\Omega_i, [g_i])}{\sigma^{N*}_{k}(\Omega,\partial^S\Omega, [g])}
$$

Let $\{g^m\}$ be a sequence of metrics on $\Omega$ defined as $g^m|_{\Omega_i}=g^m_i$. Then for large enough $m$ one has that $\sigma^N_k(\Omega,\partial^S\Omega, g^m) = \sigma^{N*}_k(\Omega,\partial^S\Omega, [g])$, since the spectrum of disjoint union is the union of spectra of each component. By definition of $\sigma^{N*}_k(\Omega, \partial^S\Omega, [g])$ we also have
$$
\sigma^{N*}_k(\Omega,\partial^S\Omega, [g])L_{g^m}(\partial^S\Omega)=\sigma^N_k(\Omega,\partial^S\Omega, g^m)L_{g^m}(\partial^S\Omega) \leq \sigma^{N*}_{k}(\Omega,\partial^S\Omega, [g]),
$$
i.e. $L_{g^m}(\partial^S\Omega) \leq 1$. Thus, one has
$$
1 \geq L_{g^m}(\partial^S\Omega) = \sum_iL_{g^m_i}(\partial^S\Omega_i)\to \frac{\sum_i \sigma^{N*}_{k_i}(\Omega_i,\partial^S\Omega_i, [g_i])}{\sigma^{N*}_{k}(\Omega,\partial^S\Omega, [g])}.
$$
Passing to the limit $m\to\infty$ yields the inequality.

{\bf Inequality $\leq$.} 

Assume the contrary, i.e. 
\begin{equation}
\label{assumption}
\sigma^{N*}_k(\Omega, \partial^S\Omega, [g]) > \max_{\sum\limits_{i=1}^s k_i=k,\,\,\,k_i>0}\,\,\sum_{i=1}^s\sigma^{N*}_{k_i}(\Omega_i, \partial^S\Omega_i, [g_i]).
\end{equation}
Consider a maximizing sequence of metrics $\{g^m\}$ of unit total length of Steklov boundary such that $\sigma^N_{k}(\Omega,\partial^S\Omega, g^m)\to\sigma^{N*}_{k}(\Omega,\partial^S\Omega, [g])$. Let $g_i^m$ be a restriction of $g^m$ to $\Omega_i$ and $d^m_i$ be the largest number satisfying $\sigma^N_{d^m_i}(\Omega_i,\partial^S\Omega_i, g_i^m)<\sigma^{N*}_k(\Omega,\partial^S\Omega,[g])$ and $\limsup_{m\to\infty}\sigma^N_{d^m_i}(\Omega_i, \partial^S\Omega_i, g_i^m)<\sigma^{N*}_k(\Omega,\partial^S\Omega, [g])$. Let $L^m_i$ denote $L_{g^m_i}(\partial^S\Omega_i)$.
Then we have  $d^m_i \leq k$ and $L^m_i \leq 1$.
Therefore, up to a choice of a subsequence one can assume that $d^m_i = d_i$ does not depend on $m$ and $L^m_i\to L_i$ as $m\to\infty$.

We claim that $\sum_i(d_i+1) \geq k+1$. Otherwise, by~\eqref{assumption} and definition of $d_i$ we have
\begin{gather*}
\sigma^{N*}_k(\Omega,\partial^S\Omega, [g])\sum_i L_i \leq\sum_i\limsup_{m\to\infty}\bar\sigma^N_{d_i+1}(\Omega_i,\partial^S\Omega_i, g^m_i) \leq\\ \leq\sum_i\sigma^{N*}_{d_i+1}(\Omega_i,\partial^S\Omega_i, [g])<\sigma^{N*}_k(\Omega,\partial^S\Omega, [g]).
\end{gather*}
Moreover,  $\sum_i L_i=1$ since $g^m$ are of unit Steklov boundary length. 
Thus, we arrive at $\sigma^{N*}_k(\Omega, \partial^S\Omega, [g])<\sigma^{N*}_k(\Omega,\partial^S\Omega, [g]),$ which is a contradiction.

Therefore, the inequality $\sum(d_i+1) \geq k+1$ holds. Since the spectrum of a union is a union of spectra, we have $$\sigma^N_k(\Omega, \partial^S\Omega, g^m)\in\bigcup_i\{\sigma_0(\Omega_i,g^m_i),\ldots,\sigma_{d_i}(\Omega_i,g^m_i)\},$$ hence 
\begin{gather*}
\sigma^{N*}_k(\Omega,\partial^S\Omega, g)=\limsup_{m\to\infty}\sigma^N_k(\Omega, \partial^S\Omega, g^m) \leq\max_i\limsup_{m\to\infty}\sigma_{d_i}(\Omega_i,g^m_i)<\\<\sigma^{N*}_k(\Omega,\partial^S\Omega, [g]).
\end{gather*}
Since $g^m$ are of unit Steklov boundary length we arrive at a contradiction.
\end{proof}

\begin{proof}[Proof of Lemma~\ref{omega_i}]
Fix indices $k_i \geq 0$ such that $\sum_{i=1}^{s'} k_i=k$ and set $I = \{i\,|\,k_i>0\}$. Let $\Omega_1 = \cup_{i\in I}\overline\Omega_i\subset \Sigma,~\partial^S\Omega_1 = \cup_{i\in I}\partial^S\Omega_i,~(\Omega_2,h) = \sqcup_{i\in I}(\overline\Omega_i,g_{\overline\Omega_i})$ and $\partial^S\Omega_2 = \sqcup_{i\in I}\partial^S\Omega_i$. One gets
\begin{gather*}
\sigma^*_k(\Sigma,[g]) \geq \sigma^{N*}_k(\Omega_1,\partial^S\Omega_1, [g]) \geq \sigma^{N*}_k(\Omega_2, \partial^S\Omega_2, [h]) \geq \\ \geq\sum_{i\in I} \sigma^{N*}_{k_i}(\Omega_i,\partial^S\Omega_i, [g])=\sum^{s'}_{i=1} \sigma^{N*}_{k_i}(\Omega_i,\partial^S\Omega_i, [g]),
\end{gather*}
where we used  in order: Proposition~\ref{subdomain}, Lemma~\ref{identity2} and Lemma~\ref{disconnected} and the fact that $\sigma^{N*}_0(\Omega_j,\partial^S\Omega_j, [g])=0$ for any $j$ in the last equality.
\end{proof}

\subsection{Proof of Lemma \ref{post2}.}\label{appendix3}
Fix $\varepsilon>0$. An application of Corollary~\ref{Neumann cor2} to a compact exhaustion of $\Sigma^\infty_j$ yields the existence of a compact set $K\subset \Sigma^\infty_j\subset \widehat{\Sigma_j^\infty}$ such that 
$$
|\sigma^*_r(\widehat{\Sigma_j^\infty},[\widehat{h_\infty}]) - \sigma^{N*}_r(K, \partial^SK, [\widehat{h_\infty}])|<\varepsilon,
$$  
where $\partial^SK=K\cap \partial\Sigma^\infty_j \neq \O$. Since $\check\Omega_j^n$ exhaust $\Sigma^\infty_j$, then for all large enough $n$ one has $K\subset\check\Omega_j^n$.
Then, by Proposition~\ref{subdomain}
$$
\sigma^{N*}_{r}(\check\Omega_j^n,\partial^S\check\Omega_j^n, [(\Psi^n)^*h_n]) \geq \sigma^{N*}_{r}(K, \partial^SK,[(\Psi^n)^*h_n]).
$$
Taking $\liminf$ of both sides in the above inequality and using Proposition~\ref{N-cont} yields
$$
\liminf_{n\to\infty}\sigma^{N*}_{r}(\check\Omega_j^n,\partial^S\check\Omega_j^n, [(\Psi^n)^*h_n]) \geq \sigma^{N*}_{r}(K,\partial^SK, [\widehat{h_\infty}])> \sigma^*_r(\widehat{\Sigma_j^\infty},[\widehat{h_\infty}])-\varepsilon.
$$
Since $\varepsilon$ is arbitrary, this completes the proof.

\bibliography{mybib}
\bibliographystyle{alpha}

\end{document}